\documentclass[reqno]{amsproc}
\usepackage[utf8]{inputenc}
\usepackage{amssymb}
\usepackage{euscript}
\usepackage{tikz}
\usepackage{extarrows}
\usepackage{mathrsfs}
\usepackage{a4wide}
\usepackage{setspace}

\newtheorem{thm}{Theorem}%[section]
\newtheorem{cor}[thm]{Corollary}
\newtheorem{lem}[thm]{Lemma}
\newtheorem{pro}[thm]{Proposition}

\theoremstyle{remark}
\newtheorem{rem}[thm]{Remark}

\newtheorem*{opq}{Question}

\theoremstyle{definition}
\newtheorem{exa}[thm]{Example}

\newcommand*{\card}[1]{\mathrm{card}(#1)}

\DeclareMathOperator{\D}{d\hspace{-0.25ex}}

\DeclareMathOperator{\dzii}{{\mathsf{Chi}}}

\DeclareMathOperator{\koo}{{\mathsf{root}}}

\DeclareMathOperator{\paa}{{\mathsf{par}}}

\DeclareMathOperator{\M}{m}

\newcommand*{\ascr}{\mathscr A}

\newcommand*{\borel}[1]{{\mathfrak B}(#1)}

\newcommand*{\cbb}{\mathbb C}

\newcommand*{\cfw}{C_{\phi,w}}

\newcommand*{\esf}{\mathsf{E}}

\newcommand*{\efw}{\mathsf{E}_{\phi,w}}

\newcommand*{\dz}[1]{{\EuScript D}(#1)}
\newcommand*{\dzi}[1]{\dzii(#1)}
\newcommand*{\dzin}[2]{\dzii^{\langle#1\rangle}(#2)}
\newcommand*{\dzn}[1]{{\EuScript D}^\infty(#1)}
\newcommand*{\ee}{\EuScript E}

\newcommand*{\Ge}{\geqslant}

\newcommand*{\hh}{\mathcal H}

\newcommand*{\hsf}{{\mathsf h}}

\newcommand*{\hfw}{{\mathsf h}_{\phi,w}}

\newcommand*{\jd}[1]{\EuScript N(#1)}

\newcommand*{\lambdab}{{\boldsymbol\lambda}}

\newcommand*{\Le}{\leqslant}

\newcommand*{\nbb}{\mathbb N}

\newcommand*{\ogr}[1]{\boldsymbol B(#1)}
\newcommand*{\ob}[1]{{\EuScript R}(#1)}

\newcommand*{\pa}[1]{\paa(#1)}

\newcommand*{\psf}{\mathsf{P}}
\newcommand*{\pfw}{\mathsf{P}_{\phi,w}}
\newcommand*{\rbb}{\mathbb R}
\newcommand*{\rbop}{{\overline{\rbb}_+}}
\newcommand*{\slam}{S_{\boldsymbol \lambda}}
\newcommand*{\smalloplus}{\raise0pt\hbox{$\scriptscriptstyle \oplus$}}

\newcommand*{\tcal}{{\mathscr T}}

\newcommand*{\zbb}{\mathbb Z}

\begin{document}
\setstretch{1.1}
\title[Centered {\em wco}'s  on $L^2$-spaces]{Centered weighted composition operators on $L^2$-spaces revisited}

\author[P.\ Budzy\'{n}ski]{Piotr Budzy\'{n}ski}
\address{Katedra Zastosowa\'{n} Matematyki, Uniwersytet Rolniczy w Krakowie, ul.\ Balicka 253c, 30-198 Kra\-k\'ow, Poland}
\email{piotr.budzynski@urk.edu.pl}

\date{\today}

\begin{abstract}
Centered weighted composition operators on $L^2$-spaces are characterized. The characterization is obtained without the assumption that the operator is a product of a multiplication and a composition operator. The concept of spectrally half-centered operators is introduced, and it is shown that unbounded weighted composition operators are spectrally half-centered provided their powers are closed and densely defined. A criteria for centered weighted shifts on directed trees of types I--IV are provided. Various examples are presented.
\end{abstract}
\maketitle

\section{Introduction}
A bounded linear operator $T$ acting on a (complex) Hilbert space $\hh$ is called {\em centered} if 
\begin{align*}
\ldots T^{2}T^{2*}, TT^{*}, T^{*}T, T^{2*}T^{2}, \ldots \text{ commute}.
\end{align*}
Centered operators generalize (classical) weighted shifts. They were introduced by Morrel in \cite{1971-cmbs-morrel}. Later, in \cite{1974-sm-morrel-muhly}, Morrel and Muhly provided a structural classification of these operators based on the properties of their phases. They showed that a centered operator is (up to a unitary equivalence) an orthogonal sum of unilateral shifts with operator weights (types I, II, and III) and a weighted translation operator on a space of vector-valued functions (type IV). The class of centered operators contains all quasinormal operators and is contained in the class of {\em weakly centered} operators. The latter class was originally studied by Campbell \cite{1975-pams-campbell}. Later, Paulsen, Pearcy, and Petrović \cite{1995-jfa-paulsen-pearcy-petrovic, 1992-jfa-petrovic} investigated centered and weakly centered operators in the context of polynomial boundedness and dilation theory. Recently, these concepts have seen renewed interest in the context of C*-modules, where Liu et al. \cite{2018-laa-liu-luo-xu} introduced the hierarchy of $n$-centered operators.

Weighted composition operators are of the form
$$f\longmapsto w \cdot (f\circ \phi),$$
with $f$ varying over an assorted function space. Classical examples include the space $\mathcal{C}(X)$ of continuous functions over a compact Hausdorff space $X$, the Hardy space $H^1(\mathbb{D})$ of analytic functions on the unit disc $\mathbb{D}$, and the space $L^2(\mu)$ of all square-integrable complex-valued functions over a $\sigma$-finite measure space $(X,\ascr, \mu)$. In this paper we investigate {\em wco}'s  (we will use this abbreviation for weighted composition operators throughout the paper) acting on the last of the mentioned spaces. 

In view of the results due to Morrel and Muhly, every centered operator is, in a sense, a {\em wco} of some general type. Interestingly, it is quite easy to provide an example of a non-centered ``classical'' {\em wco} on $L^2$-space. In fact, one can easily find $T=\cfw$ such that 
\begin{align*}
T^*T\text{ commutes with }TT^*
\end{align*}
does not hold. Operators satisfying the above are called {\em weakly centered} (some authors call them {\em binormal}). Weakly centered {\em wco}'s  have been fully characterized recently (see \cite{2025-arxiv-budzynski}).

Till the publication of \cite{2018-lnim-budzynski-jablonski-jung-stochel}, it was quite a common practice to treat {\em wco}'s  as if they all were products of multiplication operators and composition operators:
\begin{align}\label{cap01}
    \cfw = M_w C_\phi.
\end{align}
In particular, this assumption underpinned the characterization of centered {\em wco}'s attempted in \cite{2017-f-jabbarzadeh-bakhshkandi}. However, as detailed in \cite[Chapter 7]{2018-lnim-budzynski-jablonski-jung-stochel}, \eqref{cap01} does not hold in general. Consequently, characterizations relying on the Radon-Nikodym derivative of $C_\phi$ are incomplete. In fact, one can construct a bounded {\em wco} $\cfw$, equivalent to a unilateral shift of multiplicity 1, such that the associated composition operator $C_\phi$ is not even well defined (see e.g., \cite[Ex. 102]{2018-lnim-budzynski-jablonski-jung-stochel}). Such $\cfw$ is centered.  In this paper, we revisit the problem of characterizing centered {\em wco}'s in full generality, avoiding a priori restrictions on the existence of $C_\phi$. The characterization in Theorem \ref{dan01} constitutes our main result.

A significant portion of the paper is dedicated to {\em weighted shifts on directed trees} ({\em wsdt}'s). While classical weighted shifts (on $\zbb_+$ or $\zbb$) are always centered, {\em wsdt}'s offer a much richer structure. We apply our general characterization to provide a detailed classification of centered {\em wsdt}'s corresponding to the Morrel-Muhly types. In particular, we show that the presence of leaves or the lack of a root constrains the operator to specific types (I, II, or III), while the rootless and leafless case admits type I and type IV operators depending on the asymptotic behaviour of the weights. 

We illustrate our investigations with examples. For this we use {\em wsdt}'s but also {\em wco}'s induced by invertible transformations of $\rbb^\kappa$, covering both ``discrete'' and ``continuous'' measure space cases.  

We show that all {\em wco}'s  are {\em half-centered}. Recall, that an operator $T\in \ogr{\hh}$ is said to be half-centered if the set $\{T^{n*}T^{n}\colon n\in\zbb_+\}$ is composed of commuting operators. Giselsson showed that a product $M_wC_\phi$ of a bounded multiplication operator $M_w$ and a bounded composition operator $C_\phi$ is half-centered (see \cite[Proposition 2.5]{2018-oam-giselsson}). In view of the discussion above, Giselsson's result apply only to a subclass of {\em wco}'s . We generalize his result and provide its unbounded analogue in Proposition \ref{tolo1}. We also  provide a {\em wco}-specific version of his criterion \cite[Proposition 2.1]{2018-oam-giselsson} in Theorem \ref{cieczka01}.

In a subsequent research we will investigate {\em spectrally centered} and {\em spectrally $n$-weakly centered} {\em wco}'s. 
\section{Preliminaries}
We write $\zbb$, $\rbb$, and $\cbb$ for the sets of integers, real numbers, and complex numbers, respectively. By $\nbb$, $\zbb_+$, and $\rbb_+$ we denote the sets of positive integers, nonnegative integers, and nonnegative real numbers, respectively. $\rbop$ stands for $\rbb_+ \cup \{\infty\}$; $\borel{\cbb}$ denotes the $\sigma$-algebra of Borel subsets of $\cbb$.

\subsection{General operator theory}
$\ogr{\hh}$ stands for a $\mathcal{C}^*$-algebra of bounded operators on a (complex) Hilbert space $\hh$. If $A$ is an operator in $\hh$ (possibly unbounded), then $\dz{A}$, $\jd{A}$, $\ob{A}$, and $A^*$ stand for the domain, the kernel, the range, and the adjoint of $A$, respectively. %A densely defined operator $A$ in $\hh$ is {\em hyponormal} if $\dz{A}\subseteq \dz{A^*}$ and $\|A^*f\| \leqslant \|Af\|$ for all $f\in\dz{A}$.

A closed subspace $\mathcal{M}$ of $\hh$ is {\em invariant} for $A$ if  $A\big(\mathcal{M}\cap\dz{A}\big)\subseteq \mathcal{M}$ or, equivalently $P_\mathcal{M}AP_\mathcal{M}=AP_\mathcal{M}$, where $P_\mathcal{M}$ is the orthogonal projection onto $\mathcal{M}$. If $P_\mathcal{M} \dz{A}\subseteq \dz{A}$ and both $\mathcal{M}$ and $\mathcal{M}^\perp$ are invariant for $A$, then $\mathcal{M}$ is said to be {\em reducing} for $A$; this is equivalent to the following 
\begin{align*}
P_{\mathcal{M}}A\subseteq AP_{\mathcal{M}}.  
\end{align*}

Recall that for every closed densely defined operator $A$ in a Hilbert space $\hh$ there exists a (unique) partial isometry $U$ such that $\jd{U}=\jd{A}$ and $A=U|A|$, where $|A|$ is the square root of $A^*A$; the operator $U$ is called the {\em phase} of $A$ and $|A|$ is the {\em modulus} of $A$.

\subsection{Weighted composition operators}
Throughout the paper we assume that $(X,\ascr, \mu)$ is a $\sigma$-finite measure space, $\phi$ is an $\ascr$-measurable {\em transformation} of $X$ (i.e., an $\ascr$-measurable mapping $\phi\colon X\to X$) and $w\colon X\to\cbb$ is $\ascr$-measurable.

Let $\mu_w$ and $\mu_w\circ\phi^{-1}$ be measures on $\ascr$ defined by $\mu_w(\sigma)=\int_\sigma|w|^2\D\mu$ and $\mu_w\circ \phi^{-1}(\sigma)=\mu_w\big(\phi^{-1}(\sigma)\big)$ for $\sigma\in\ascr$. Assume that $\mu_w\circ\phi^{-1}$ is absolutely continuous with respect to  $\mu$. Then the operator 
\begin{align*}
\cfw \colon L^2(\mu) \supseteq \dz{\cfw} \to L^2(\mu)    
\end{align*}
given by
\begin{align*}
\dz{\cfw} = \{f \in L^2(\mu) \colon w \cdot (f\circ \phi) \in L^2(\mu)\},\quad
\cfw f  = w \cdot (f\circ \phi), \quad f \in \dz{\cfw},
\end{align*}
is well defined (see \cite[Proposition 7]{2018-lnim-budzynski-jablonski-jung-stochel}); here, as usual, $L^2(\mu)$ stands for the complex Hilbert space of all square $\mu$-integrable $\ascr$-measurable complex functions on $X$ (with the standard inner product). $\cfw$ is called a {\em weighted composition operator} (abbreviated to {\em wco} throughout the paper). By the Radon-Nikodym theorem (cf.\ \cite[Theorem 2.2.1]{ash}), there exists a unique (up to a set of $\mu$-measure zero) $\ascr$-measurable function $\hfw\colon X \to \rbop$ such that
 \begin{align*}
\mu_w \circ \phi^{-1}(\varDelta) = \int_{\varDelta} \hfw \D \mu, \quad
\varDelta \in \ascr.
\end{align*}
As it follows from \cite[Theorem 1.6.12]{ash} and \cite[Theorem 1.29]{rud}, for every $\ascr$-measurable function $f \colon X \to \rbop$ (or for every $\ascr$-measurable function $f\colon X \to \cbb$ such that $f\circ \phi \in L^1(\mu_w)$), we have
\begin{align*}
\int_X f \circ \phi \D\mu_w = \int_X f \, \hfw \D \mu.
\end{align*}
In particular, this implies that $\dz{\cfw}=L^2\big((1+\hfw)\D\mu\big)$. Recall that $\cfw$ is a bounded operator defined on the whole of $L^2(\mu)$ if and only if $\hfw$ is $\mu$-essentially bounded and, if this is the case, $\|\cfw\|^2=\|\hfw\|_{L^\infty(\mu)}$ (see \cite[Proposition 8]{2018-lnim-budzynski-jablonski-jung-stochel}).

Assume that $\cfw$ is densely defined. Then, for a given $\ascr$-measurable function $f\colon X\to \rbop$ (or $f\colon X\to \cbb$ such that $f\in L^p(\mu_w)$ with some $1\leqslant p<\infty$) one may consider $\efw(f)$, the conditional expectation of $f$ with respect to the $\sigma$-algebra $\phi^{-1}(\ascr)$ and the measure $\mu_w$ (see \cite[Section 2.4]{2018-lnim-budzynski-jablonski-jung-stochel}). Also, there is a unique (up to sets of $\mu$-measure zero) $\ascr$-measurable function $g$ on $X$ such that $g=0$ a.e. $[\mu]$ on $\{\hfw=0\}$ and $\efw(f)=g\circ\phi$ a.e. $[\mu_w]$. We will denote this function by $\efw(f)\circ\phi^{-1}$. Recall that $\efw$ can be regarded as a linear contraction on $L^p(\mu_w)$, $p\in[1,\infty]$, which leaves invariant the convex cone $L^p_+(\mu_w)$ of all $\rbb_+$-valued members of $L^2(\mu_w)$; moreover, $\efw$ is an orthogonal projection on $L^2(\mu_w)$.

Considering $w\equiv 1$ leads to a large subclass of weighted composition operators, the so-called composition operators. Namely, assuming that $\phi\colon X\to X$ is $\ascr$-measurable and satisfies $\mu\circ\phi^{-1}\ll\mu$, the operator $C_{\phi, 1}$ is well defined. We call it the {\em composition operator} induced by $\phi$; for brevity we use the notation $C_\phi:=C_{\phi,1}$. The corresponding Radon-Nikodym derivative $\mathsf{h}_{\phi,1}$ and the conditional expectation $\esf_{\phi,1}$ (see the definition below) are denoted by $\mathsf{h}_\phi$ and $\esf_\phi$, respectively. We caution the reader that $C_\phi$ might not be well defined even if $\cfw$ is an isometry. In particular, if this is the case, $\hsf_\phi$ and $\esf_{\phi}$ do not exist whereas $\hfw$ and $\efw$ do (see \cite[pg. 71]{2018-lnim-budzynski-jablonski-jung-stochel}).

Now we introduce some useful notation. First, we set
\begin{align*}
w_0=\chi_X \quad \text{and}\quad w_{n+1}=\prod_{j=0}^n w\circ\phi^j \text{ for }n\in\zbb_+.
\end{align*}
In particular, $w_1=w$. It is known (see \cite[Lemma 26]{2018-lnim-budzynski-jablonski-jung-stochel}) that for every $n\in\nbb$, $C_{\phi^n,w_n}$ is well defined and $\cfw^n\subseteq C_{\phi^n,w_n}$. 
For brevity, whenever this leads to no confusion, we use the following: 
\begin{align*}
\hsf_n=\hsf_{\phi^n, w_n},\quad\esf_n=\esf_{\phi^n, w_n},\quad n\in\nbb.
\end{align*}
(Recall that, $\esf_{n}$ exists whenever $\hsf_n<\infty$ a.e. $[\mu]$.) A version of \cite[Lemma 26]{2018-lnim-budzynski-jablonski-jung-stochel} will be helpful in our further consideration (the proof is essentially the same but we provide it for completeness).
\begin{lem}\label{lucek3}
Let $n\in\nbb$, $k\in\zbb_+$. Assume that $C_{\phi, w}$ is well defined and $C_{\phi^n, {w}_n}$ is densely defined. Then the following is satisfied
\begin{align}\label{cap02}
    \hsf_{\phi^{n+k},w_{n+k}}=\esf_{\phi^n,w_n}(\hsf_{\phi^k, w_k})\circ\phi^{-n}\cdot \hsf_{\phi^{n},w_{n}}\text{ a.e. }[\mu]
\end{align}
In particular, if $C_{\phi, w}\in\ogr{L^2(\mu)}$, then the following conditions hold:
\begin{align}\label{avici01}
\hsf_{\phi^{m+1},w_{m+1}}\circ \phi=\esf_{\phi, w}\big(\hsf_{\phi^m, w_m}\big) \cdot \hsf_{\phi, w}\circ \phi\quad \text{a.e. $[\mu_{w}]$},\quad m\in\zbb_+.
\end{align}
\end{lem}
\begin{proof}
In view of \cite[Lemmma 26]{2018-lnim-budzynski-jablonski-jung-stochel}, $\hsf_{{n+k}}$, $\hsf_{{n}}$, and $\hsf_{{k}}$ exist. For every $\varDelta\in\ascr$ we have\allowdisplaybreaks
\begin{align*}
\mu_{w_{n+k}}\bigg(\big(\phi^{n+k}\big)^{-1}(\varDelta)\bigg)&=\int \chi_{\phi^{-n}(\varDelta)}\circ\phi^k \cdot |w_n|^2\circ\phi^k \D\mu_{w_k}=\int \chi_{\phi^{-n}(\varDelta)}\hsf_{k} \D\mu_{w_n}\\
&=\int \chi_{\phi^{-n}(\varDelta)} \esf_{\phi^{n}
w_n}(\hsf_{\phi^k,w_k}) \D\mu_{w_n}=\int_{\varDelta} \esf_{n}(\hsf_{k})\circ\phi^{-n} \cdot \hsf_{{n}}\D\mu.
\end{align*}
This completes the proof of \eqref{cap02}. Obviously, \eqref{avici01} follows from \eqref{cap02} and \cite[Lemma 5 and (2.10)]{2018-lnim-budzynski-jablonski-jung-stochel}.
\end{proof}
We also set:
\begin{align}\label{djo01}
w_{\phi,n}=\frac{w}{\sqrt{\hfw\circ\phi}}\cdot\frac{w\circ \phi}{\sqrt{\hfw\circ \phi^2}}\cdots\frac{w\circ\phi^{n-1}}{\sqrt{\hfw\circ\phi^n}},\quad n\in\nbb,
\end{align}
and
\begin{align}\label{djo02}
w_{n,\phi}=\frac{w_n}{\sqrt{\hsf_{\phi^n,w_n}\circ\phi^n}},\quad n\in\nbb.
\end{align}
In particular, $w_{1,\phi}=\frac{w}{\sqrt{\hfw\circ\phi}}$.

In view of \cite[Theorem 18]{2018-lnim-budzynski-jablonski-jung-stochel}, assuming $\cfw$ is densely defined, the phase of $\cfw$ is equal to $C_{\phi,w_{1,\phi}}$ and the modulus of $\cfw$ is equal to $M_{\sqrt{\hfw}}$ (we use here a somewhat different notation than in \cite{2018-lnim-budzynski-jablonski-jung-stochel}). Here and later on, for any given an $\ascr$-measurable function $g\colon X\to \cbb$, $M_g$ denotes the operator of multiplication by $g$ acting in $L^2(\mu)$.

\section{Main results}
Our first result is an analogue of \cite[Proposition 2.5]{2018-oam-giselsson}. Specifically, we prove that every ``power closed'' weighted composition operators whose $C^\infty$-vectors are dense in $L^2(\mu)$, is spectrally half-centered. In our considerations we are not restricted to {\em wco}'s  that are products of multiplication operators and composition operators. In fact, we carry out without assuming that the operators are bounded if this does not seem to be necessary. For this we first generalize the definition of half-centered operators to the unbounded setting. 

Let $T$ be an operator in $\hh$ such that $T^n$ is densely defined and closed for every $n\in\nbb$. This implies that every $T^{n*}T^n$ is selfadjoint. In particular, there exists a spectral measure $E_{1,n}$ of $T^{n*}T^n$. We say that $T$ is {\em spectrally half-centered} if and only if
\begin{align}\label{rzaska01}
E_{1,n}(\sigma)E_{1,m}(\omega)=E_{1,m}(\omega)E_{1,n}(\sigma),\quad \sigma,\omega\in\borel{\cbb},\ n,m \in \nbb.
\end{align}
Clearly, for $T\in\ogr{\hh}$, $T$ is half-centered if and only if $T$ is spectrally half-centered.

Recall that $\dzn{\cfw}:=\bigcap_{n=1}^\infty \dz{\cfw^n}$.
\begin{pro}\label{tolo1}
Assume that $\cfw^n$ is closed for every $n\in\nbb$ and $\overline{\dzn{\cfw}}=L^2(\mu)$. Then $\cfw$ is spectrally half-centered.
\end{pro}
\begin{proof}
In view of \cite[Theorem 47 and Lemma 44]{2018-lnim-budzynski-jablonski-jung-stochel}, $\cfw^n=C_{\phi^n,w_n}$ for every $n\in\nbb$. Thus, using \cite[Theorem 18]{2018-lnim-budzynski-jablonski-jung-stochel} we get
\begin{align}\label{rzaska02}
\cfw^{n*}\cfw^n=M_{\hsf_{n}},\quad n\in\nbb.
\end{align}
Hence, the spectral measure $E_{1,n}$ of $\cfw^{n*}\cfw^n$ is given by (see \cite[Ex. 5.3, pg. 93]{schmudgen})
\begin{align*}
E_{1,n}(\sigma)g=\chi_{\varDelta_{\sigma,n}} g,\quad g\in L^2(\mu),\ \sigma\in\borel{\cbb},
\end{align*}
with $\varDelta_{\sigma,n}=(\hsf_{{\phi^n, w_n}})^{-1}(\sigma)$. Clearly, \eqref{rzaska01} is satisfied for such $E_{1,n}$'s. Thus, $T$ is spectrally half-centered.
\end{proof}
\begin{rem}
In view of \eqref{rzaska02}, the set $\{\cfw^{n*}\cfw^n\colon n\in \nbb\}$ consists of operators that commute pointwise on $\dzn{\cfw}$. This means that {\em wco}'s  are, in a sense, {\em pointwise half-centered}. If $\cfw\in \ogr{L^2(\mu)}$, then $\dzn{\cfw}=L^2(\mu)$ and thus $\cfw$ is half-centered in a usual sense.

It is worth recalling that unbounded multiplication operators does not commute in general. For given weight functions $u, v\colon X\to\cbb$, $M_uM_v=M_vM_u$ if and only if $\dz{M_uM_v}=\dz{M_vM_u}$. Since $\dz{M_uM_v}=L^2\big((1+|v|^2+|uv|^2)\D \mu\big)$ and $\dz{M_vM_u}=L^2\big((1+|u|^2+|uv|^2)\D \mu\big)$, using \cite[Lemma 12.3]{2014-ampa-budzynski-jablonski-jung-stochel}, we see that $M_uM_v=M_vM_u$ if and only if there exist positive constants $c_1, c_2$ such that
\begin{align*}
c_1\big(1+|v|^2+|uv|^2\big)\leqslant 1+|u|^2+|uv|^2 \leqslant c_2 \big(1+|v|^2+|uv|^2\big) \quad \text{a.e. $[\mu]$}.
\end{align*}
The above fact is surely of folklore type.
\end{rem}
Giselsson proved in \cite[Proposition 2.1]{2018-oam-giselsson} that a half-centered operator $T\in\ogr{\hh}$ is centered if and only if $\jd{T^*}$ is invariant for every $T^{*n}T^n$, $n\in\nbb$. Since each $T^{*n}T^n$ is selfadjoint, $\jd{T^*}$ is invariant if and only if it is reducing, i.e., $\jd{T^*}^\perp$ is invariant for $T^{*n}T^n$ as well. We may rephrase the essential part of the Giselsson's criterion using language of commutativity: a half-centered operator $T\in\ogr{\hh}$ satisfying $PT^{*n}T^n=T^{*n}T^nP$ for every $n\in\nbb$, where $P$ is the orthogonal projection onto $\jd{T^*}$ is centered. This motivates Theorem \ref{cieczka01} below (its connection to centered {\em wco}'s  will be revealed later in the paper). For our convenience we introduce some notation: $\pfw$ stands for the operator on $L^2(\mu)$ given by 
\begin{align*}
\pfw f=w\efw(f_w),\quad f\in L^2(\mu).
\end{align*}
It is known (see \cite[Lemma 4.2]{2020-mn-benhida-budzynski-trepkowski}) that $\pfw$ is an orthogonal projection.

\begin{thm}\label{cieczka01}
Assume the following condition:
\begin{itemize}
\item[(i)] for every $n\in\nbb$, $\cfw^n$ is densely defined and closed.
\end{itemize}
Then the following conditions are equivalent:
\begin{itemize}
\item[(ii)] for every $n\in\nbb$, $\pfw M_{\hsf_{\phi^n,w_n}}\subseteq M_{\hsf_{\phi^n,w_n}} \pfw$,
\item[(iii)] for every $n\in\nbb$, $\efw\big(\hsf_{\phi^n,w_n}\big)= \hsf_{\phi^n,w_n}$ a.e. $[\mu_w]$. 
\end{itemize}
\end{thm}
\begin{proof}
(ii)$\Rightarrow$(iii) We first reduce the proof to the case $w\geqslant 0$ a.e. $[\mu]$. For this, we use \cite[Proposition 115]{2018-lnim-budzynski-jablonski-jung-stochel}. It implies that $C_{\phi, |w|}$ is well defined and $\cfw=M_sC_{\phi, |w|}$, where
\begin{align*}
    s(x)=\left\{ \begin{array}{cc}
         1&  \text{if } w(x)=0,\\
         \frac{w(x)}{|w(x)|}& \text{if } w(x)\neq0. 
    \end{array}\right.
\end{align*}
Moreover, $\mu_w=\mu_{|w|}$, $\hfw=\hsf_{\phi,|w|}$, and $\efw =\esf_{\phi, |w|}$. Analogous formulas may be written with $\phi^n$ and $|w_n|$ (and appropriately defined $s_n$), for every $n\in\nbb$. Note that $M_{s_n}\in\ogr{L^2(\mu)}$ is bounded from below. 

Condition (i) implies that $\dzn{\cfw}$ is dense in $L^2(\mu)$ and $\cfw^n=C_{\phi^n, w_n}$ for every $n\in\nbb$. Therefore, $C_{\phi^n, |w_n|}$ is densely defined for every $n\in\nbb$ and $\overline{\dzn{C_{\phi^n, |w_n|}}}=L^2(\mu)$. Also, $\cfw^{n*}\cfw^n=M_{\hsf_{\phi^n,w_n}}=M_{\hsf_{\phi^n,|w_n|}}$, $n\in\nbb$. 

Observe that $\jd{\pfw}=\jd{|\cfw^*|}=\jd{\cfw^*}$. Indeed, if $f\in \jd{\pfw}$, then $f\in L^2(\mu)$ and $w\efw(f_f)=0$ a.e. $[\mu]$. The latter implies that $w\sqrt{\hfw\circ\phi} \efw(f_w)=0$ a.e. $[\mu]$. Since, in view of \cite[Theorem 18]{2018-lnim-budzynski-jablonski-jung-stochel}, $\dz{|\cfw^*|}=\{f\in L^2(\mu)\colon w\sqrt{\hfw\circ\phi}\, \efw(f_w)\in L^2(\mu)\}$, we deduce that $f\in\dz{\cfw^*}$ and $|\cfw^*|f=0$. Hence, $\jd{\pfw}\subseteq\jd{|\cfw^*|}$. The opposite inclusion can be proved essentially the same way.

Set $\ee_{\phi, |w|}=\jd{\psf_{\phi,|w|}}^\perp$ and $\ee_{\phi, w}=\jd{\psf_{\phi,w}}^\perp$. Note that
\begin{align*}
\ee_{\phi, w}^\perp=\jd{\cfw^*}=\jd{(M_s C_{\phi, |w|})^*}=\jd{C_{\phi, |w|}^*M_{\bar s}}.
\end{align*}
Since $M_{\bar s}$ is injective, we deduce that
\begin{align}\label{pioro01}
\ee_{\phi, |w|}^\perp =M_{\bar s} \ee_{\phi, w}^\perp\text{ and } \ee_{\phi, |w|}= M_{s^{-1}}\ee_{\phi, w}.
\end{align}

We now show that $\ee_{\phi, |w|}$ is reducing for all $M_{\hsf_{\phi^n,|w_n|}}$'s. Condition (ii) implies that $\ee_{\phi, w}$, and $\ee^\perp$ are invariant for all $M_{\hsf_{\phi^n,w_n}}$'s. Let $\hsf$ denote any of the $\hsf_{\phi^n, w_n}$, $n\in\nbb$. Let $f\in \ee_{\phi, |w|}\cap\dz{M_\hsf}$. Then, in view of \eqref{pioro01},  $M_{s_n}f\in \ee_{\phi, w}\cap \dz{M_\hsf}$ (use also $|s_n|=1$). Thus, by (ii), $\hsf M_sf=s\hsf f\in \ee_{\phi, w}$. Consequently, by \eqref{pioro01}, we have
\begin{align*}
\hsf_{\phi^n, |w_n|} f=\hsf f =\frac{1}{s}\big(s\hsf f\big)\in M_{s^{-1}}\ee_{\phi, w}=\ee_{\phi, |w|}.
\end{align*}
This proves that $\ee_{\phi, |w|}$ is invariant for $M_\hsf$. In a similar fashion we show that $\ee_{\phi, |w|}^\perp$ is invariant for $M_\hsf$ as well. Since
\begin{align*}
\psf_{\phi, |w|}f=|w|\esf_{\phi, |w|}\big(f_{|w|})=\bar s w \efw \big(\tfrac{1}{\bar s}f_w\big) = M_{\bar s}\psf_{\phi, w} M_{\bar s^{-1}} f,\quad f\in L^2(\mu),
\end{align*}
we may show that for every $f\in \dz{M_h}$, $P_{\phi, |w|}f\in \dz{M_\hsf}$. This means that $\ee_{\phi, |w|}$ is reducing for $M_\hsf$, and thus $\ee_{\phi, |w|}$ is reducing for all $M_{\hsf_{\phi^n, |w_n|}}$'s. In another words, $\psf_{\phi, |w|}M_{\hsf_{\phi^n, |w_n|}}\subseteq M_{\hsf_{\phi^n, |w_n|}} P_{\phi, |w|}$ for every $n\in\nbb$. Summing up, (i) and (ii) are satisfied with $\phi$ and $|w|$ in place of $\phi$ and $w$. Note also, that (iii) is satisfied with $\phi$ and $w$ if and only if it is satisfied with $\phi$ and $|w|$. Therefore, without loosing generality, we may assume that $w\geqslant0$.

Condition (ii) implies that
\begin{align*}
w \hsf \efw (f_w)= w\efw(\hsf f_w) \text{ a.e. $[\mu]$}, f\in \dz{M_\hsf}. 
\end{align*}
(Again, $\hsf$ denotes any of the Radon-Nikodym derivatives $\hsf_{\phi^n, w_n}$, $n\in\nbb$). Let $L^2_+(\mu)=\{f\in L^2(\mu) \colon f\geqslant 0\}$ and ${\EuScript D}_+(M_\hsf)= L^2_+(\mu)\cap \dz{M_\hsf}$. Clearly, ${\EuScript D}_+(M_\hsf)$ is dense in $L^2_+(\mu)$. Also, we see that
\begin{align}\label{pioro03}
w \hsf \efw (f_w)= w\efw(\hsf f_w) \text{ a.e. $[\mu]$}, f\in {\EuScript D}_+(M_\hsf),
\end{align}
and all the functions appearing in \eqref{pioro03} are $\rbb_+$-valued, which enables us to use the monotone convergence theorem. Let $f\in L^2_+(\mu)$. There exists a sequence $\{f_k\}\subseteq {\EuScript D}_+(M_\hsf)$ such that $f_k\nearrow f$ as $k\to +\infty$. Thus  $w \efw \big((\hsf f_k)_w\big)\nearrow w \efw \big((\hsf f)_w\big)$ and $w\hsf \efw \big((f_k)_w\big)\nearrow w \hsf\efw \big((f_k)_w\big)$ as $k\to+\infty$ (use \cite[(A.6) in Appendix A]{2018-lnim-budzynski-jablonski-jung-stochel}). This implies that the equality in \eqref{pioro03} is satisfied for any $f\in L^2_+(\mu)$. Using $\sigma$-finiteness of $\mu$ and standard measure theoretic arguments we deduce that 
\begin{align*}
\hsf \efw (f)= \efw(\hsf f) \text{ a.e. $[\mu_w]$}
\end{align*}
holds for every $\ascr$-measurable $f\colon X\to\rbb_+$. Clearly, this implies (iii).

(iii)$\Rightarrow$(ii) Fix $n\in\nbb$ and denote $\hsf=\hsf_{\phi^n,w_n}$. We first show that $\psf_{\phi,w}\dz{M_{\hsf}}\subseteq \dz{M_{\hsf}}$. Take $f\in \dz{M_\hsf}=L^2\big((1+\hsf)\D\mu\big)$. Clearly, $\psf_{\phi, w}f\in L^2(\mu)$. Since $\big|\efw(f_w)\big|^2\leqslant \efw\big(|f_w|^2\big)$ a.e. $[\mu_w]$ (see \cite[Theorem A.4]{2018-lnim-budzynski-jablonski-jung-stochel}) and $f\in L^2(\hsf\D\mu)$, we get by (iii):
\begin{align*}
\int_X |\psf_{\phi, w} f|^2 \hsf\D\mu&= \int_X |w|^2 \big|\efw (f_w)\big|^2 \hsf\D\mu\leqslant \int_X |w|^2 \efw \big(|f_w|^2\big) \hsf\D\mu\\
&= \int_X \efw \big(\hsf|f_w|^2\big)\D\mu_w= \int_X \hsf|f_w|^2\D\mu_w\leqslant \int_X \hsf|f|^2\D\mu<+\infty.
\end{align*}
Thus $\pfw f\in L^2\big((1+\hsf)\D\mu\big)=\dz{M_\hsf}$. Consequently, $\psf_{\phi,w}\dz{M_{\hsf}}\subseteq \dz{M_{\hsf}}$. That the equality $\pfw M_\hsf f=M_\hsf \pfw f$ is valid for all $f\in\dz{M_\hsf}$ follows immediately from \cite[(A.13) in Appendix A]{2018-lnim-budzynski-jablonski-jung-stochel} and the fact that $f_w$ and $\hsf f_w$ belong to $L^2(\mu_w)$. This completes the proof.
\end{proof}
\begin{rem}
Half-centered operators and centered operators are related to weakly centered operators. Spectrally weakly-centered {\em wco}'s  (unbounded!) have been characterized in \cite{2025-arxiv-budzynski} via the following condition
\begin{align*}
\hfw =\efw (\hfw)\text{ a.e. $[\mu]$}.
\end{align*}
In view of Theorem \ref{cieczka01}, any operator satisfying conditions (i) and (ii) is spectrally weakly-centered. Inspecting the proof one can deduce that conditions:
\begin{enumerate}
\item[(i)$_1$] $\cfw$ is densely defined,
\item[(ii)$_1$] $\pfw M_{\hfw} \subseteq M_{\hfw} \pfw$,
\end{enumerate}
are sufficient (and necessary) for $\cfw$ to be spectrally weakly centered.
\end{rem}
The above considerations raise an interesting problem of whether unbounded centered operators can be defined in a meaningful way and, if so, whether unbounded centered {\em wco}'s  can be characterized via condition (iii) of Theorem \ref{cieczka01}. %We will address this problem in the last part of the paper. 
Below we show that the above holds true for bounded {\em wco}'s .

We are ready to present the main result of the paper.
\begin{thm}\label{dan01}
Assume $\cfw\in\ogr{L^2(\mu)}$. The following conditions are equivalent:
\begin{enumerate}
\item[(A)] $\cfw$ is centered,
\item[(B)] for every $n\in\nbb$, $C_{\phi^n, w_{n,\phi}}=C_{\phi^n,w_{\phi, n}}$, with $ w_{n,\phi}$ and $ w_{\phi, n}$ defined in \eqref{djo01} and \eqref{djo02},
\item[(C)] for every $n\in\nbb$, $\hfw\circ\phi \cdot \hfw\circ\phi^{2}\cdots\hfw\circ\phi^{n}=\hsf_{\phi^n,w_n}\circ\phi^n$ a.e. $[\mu_{w_n}]$,
\item[(D)] for every $n\in\nbb$, $\hsf_{\phi^n, w_n}=\efw(\hsf_{\phi^n, w_n})$ a.e. $[\mu_w]$,
\item[(E)] for every $n\in\nbb$, $\hsf_{\phi^{n+1},w_{n+1}}\circ \phi=\hsf_{\phi^n,w_n}\cdot \hfw\circ\phi$ a.e. $[\mu_w]$,
\item[(F)] for every $n\in\nbb$, $\hsf_{\phi, w}=\esf_{\phi^n, w_n}\big(\hsf_{\phi,w}\big)$ a.e. $[\mu_{ w_n}]$,
\item[(G)] for every $n\in\nbb$, $\esf_{\phi^n, w_n}\big(\hsf_{\phi, w}\big)\circ \phi =\esf_{\phi^{n+1}, w_{n+1}}\big(\hsf_{\phi,w}\circ\phi \big)$ a.e. $[\mu_{ w_{n+1}}]$,
\item[(H)] for every $n,k\in\nbb$, $\hsf_{\phi^n, w_n}=\esf_{\phi^k, w_k}\big(\hsf_{\phi^n,w_n}\big)$ a.e. $[\mu_{w_k}]$,
\end{enumerate}
\end{thm}
\begin{proof}
(A) $\Leftrightarrow$ (B) Fix $n\in\nbb$. In view of \cite[Theorem 18 and Lemma 44]{2018-lnim-budzynski-jablonski-jung-stochel}, $\cfw^n=C_{\phi^n, {w}_n}$. Thus, by \cite[Theorem 18]{2018-lnim-budzynski-jablonski-jung-stochel}, the phase $U_n$ in the polar decomposition $\cfw^n=U_n|\cfw^n|$ is given by
\begin{align}\label{wallace01}
U_n=C_{\phi^n,{w}_{n,\phi}}.
\end{align}
On the other hand, again by \cite[Theorem 18 and Lemma 26]{2018-lnim-budzynski-jablonski-jung-stochel}, we have
\begin{align}\label{wallace02}
U_1^n=C_{\phi,w_{1,\phi}}^n=C_{\phi^n,w_{\phi,n}}.
\end{align}
Since an operator $T\in\ogr{\hh}$ is centered if and only if $V_1^k=V_k$ for every $k\in\nbb$, where $V_k$ is the partial isometry part of the polar decomposition $T^k=V_k|T^k|$ (see \cite[Theorem 1]{1974-sm-morrel-muhly} and \cite[Theorem 3.2]{2004-ieot-ito-yamazaki-yanagida}), using \eqref{wallace01} and \eqref{wallace02} we get the claim.

(B) $\Leftrightarrow$ (C) Since characteristic functions of finite $\mu$ measure sets are dense in $L^2(\mu)$, (ii) holds if and only if 
\begin{align*}
    \chi_{\phi^{-n}(\varDelta)} w_n \sqrt{\hsf_{n}\circ\phi^n}=\chi_{\phi^{-n}(\varDelta)} w_n\sqrt{\hsf_1\circ\phi\cdots \hsf_1\circ\phi^n}\quad \text{a.e. }[\mu],\ \varDelta\in\ascr,\quad n\in\nbb,
\end{align*}
or equivalently,
\begin{align}\label{stachu3}
    \chi_{\phi^{-n}(\varDelta)} {\hsf_{n}\circ\phi^n}=\chi_{\phi^{-n}(\varDelta)}{\hsf_1\circ\phi\cdots \hsf_1\circ\phi^n}\quad \text{a.e. }[\mu_{ w_n}],\ \varDelta\in\ascr,\quad n\in\nbb.
\end{align}
Clearly, the above is satisfied if and only if so is \eqref{stachu3}. Hence (ii) and (iii) are equivalent.

(A) $\Leftrightarrow$ (D) By Proposition \ref{tolo1}, $\cfw$ is half-centered. Thus, in view of \cite[Proposition 2.1]{2018-oam-giselsson}, $\cfw$ is centered if and only if the following inclusion holds
\begin{align}\label{riot1}
    \cfw^{k*}\cfw^k \jd{\cfw^*}\subseteq \jd{\cfw^*},\quad k\in \nbb.
\end{align}
Since $\cfw\in\ogr{L^2(\mu)}$, $\cfw^{n*}\cfw^n=M_{\hsf_{n}}$ for all $n\in\nbb$. This and $\jd{\cfw^*}^\perp=\pfw(L^2(\mu))$ implies that \eqref{riot1} is equivalent to 
\begin{align*}
    \pfw M_{\hsf_{n}}= M_{\hsf_{n}} \pfw,\quad n\in \nbb.
\end{align*}
It suffices now to apply Theorem \ref{cieczka01}.

(D) $\Leftrightarrow$ (E) This follows immediately from Lemma \ref{lucek3} and the fact that $\hsf_1\circ\phi>0$ a.e. $[\mu_w]$ (see \cite[Lemma 6]{2018-lnim-budzynski-jablonski-jung-stochel}).

(D) $\Rightarrow$ (F) By Lemma \ref{lucek3}, we get
\begin{align*}%\label{adhd4}
\hsf_{{1+1}}\circ\phi = \hsf_{{1}}\hsf_1\circ\phi \quad \text{a.e. }[\mu_w].
\end{align*}
Since $\esf_1\big(\hsf_{{1+1}}\big)=\hsf_{{1+1}}$ a.e. $[\mu_w]$ and $\esf_1\big(\hsf_{1}\big)=\hsf_{1}$ a.e. $[\mu_w]$, we deduce that $\esf_{2}\big(\hsf_{{1}}\big)=\hsf_{{1}}$ a.e. $[\mu_{w_2}]$. Using Lemma \ref{lucek3} again, we obtain
\begin{align*}%\label{adhd41}
\hsf_{{1+2}}\circ\phi^2 = \hsf_{{1}}\hsf_{2}\circ\phi^2 \quad \text{a.e. }[\mu_{w_2}].
\end{align*}
Since $\esf_1\big(\hsf_{3}\big)=\hsf_{3}$ a.e. $[\mu_w]$ and $\esf_1\big(\hsf_{2}\big)=\hsf_{2}$ a.e. $[\mu_w]$, we get $\esf_{3}\big(\hsf_{1}\big)=\hsf_{1}$ a.e. $[\mu_{w_3}]$. Repeating the argument we deduce (F).

(F) $\Rightarrow$ (D) For $n\in\nbb$, consider the following conditions
\begin{enumerate}
\item[(D$_n$)] for every $j\in J_n$, $\hsf_{\phi^j, w_j}=\esf_{\phi, w} \big(\hsf_{\phi^j, w_j}\big)$ a.e. $[\mu_{w}]$,
\item[(F$_n$)] for every $j\in J_n$, $\hsf_{\phi, w}=\esf_{\phi^j, w_j}\big(\hsf_{\phi,w}\big)$ a.e.$[\mu_{ w_j}]$
\end{enumerate}
Clearly (D$_1$) and (F$_1$) are equivalent (D$_{k+1}$) implies (D$_k$), and (F$_{k+1}$) implies (F$_k$). Hence we can use induction. Fix $k\in J_{n-1}$. Assume that condition (F$_{k+1}$) is satisfied and that (F$_k$) implies (D$_k$). Using Lemma \ref{lucek3} and induction hypothesis we get
\begin{align}\label{govinda01}
|w|^2\cdot\hsf_{k+1}=|w|^2\cdot \esf_{k}(\hsf_1)\circ\phi^{-k}\cdot \esf_1(\hsf_k)\text{ a.e. $[\mu]$}.
\end{align}
We also have
\begin{align}\label{govinda02}
|w_k|^2\cdot\hsf_1=|w_k|^2\cdot \big(H_{k+1}\circ\phi^k\big)\text{ a.e. $[\mu]$},
\end{align}
with $\chi_{\{\hsf_k>0\}}H_{k+1}=H_{k+1}=\esf_{k}(\hsf_1)\circ \phi^{-k}$ a.e. $[\mu]$, and
\begin{align}\label{govinda03}
|w_{k+1}|^2\cdot \hsf_1=|w_{k+1}|^2\cdot \big(G_{k+1}\circ\phi^{k+1}\big)\text{ a.e. $[\mu]$}.
\end{align}
with some $\ascr$-measurable function $G\colon X\to\rbb_+$ (this follows from (F$_{k+1}$)). Comparing \eqref{govinda02} and \eqref{govinda03}, we get from \cite[Lemma 5]{2018-lnim-budzynski-jablonski-jung-stochel}
\begin{align}\label{govinda04}
|w|^2\cdot\chi_{\{\hsf_{k}>0\}}\cdot \big(G_{k+1}\circ\phi\big)=|w|^2\cdot \chi_{\{\hsf_{k}>0\}}\cdot H_{k+1}= |w|^2\cdot H_{k+1}\text{ a.e. $[\mu]$}.
\end{align}
Since (D$_k$) is satisfied, $\{\hsf_k\}\cap\{w\neq0\}\in \phi^{-1}(\ascr)\cap\{w\neq 0\}$, which yields $|w|^2\cdot \chi_{\{\hsf_k>0\}}=|w|^2\cdot \big(K_{k+1}\circ\phi\big)$ a.e. $[\mu]$, where $K_{k+1}=\chi_{\varDelta_{k+1}}$ with some $\varDelta\in\ascr$. Plugging this into \eqref{govinda04} we get
\begin{align*}
|w|^2\cdot H_{k+1}= |w|^2\cdot \big(K_{k+1}\circ\phi\big) \cdot \big(G_{k+1}\circ\phi\big) \text{ a.e. $[\mu]$}.
\end{align*}
Using the above and \eqref{govinda01} we see that
\begin{align*}
|w|^2\cdot\hsf_{k+1}=|w|^2\cdot \big(K_{k+1}\circ\phi\big) \cdot \big(G_{k+1}\circ\phi\big)\cdot \esf_1(\hsf_k)\text{ a.e. $[\mu]$}.
\end{align*}
Clearly, this implies (D$_{k+1}$). Using induction we get (D).

(G) $\Rightarrow$ (F) Fix $n \in \nbb$ and let $E_n := \esf_{\phi^n, w_n}(\hsf_{\phi, w})$. We have
\begin{align*}
    E_n \circ \phi = \esf_{\phi^{n+1}, w_{n+1}}(\hsf_{\phi, w} \circ \phi) \quad \text{a.e. } [\mu_{w_{n+1}}].
\end{align*}
Since $\cfw\in\ogr{L^2(\mu)}$ is a bounded operator, $\hsf_{\phi, w} \in L^\infty(\mu)$, which ensures that $E_n \in L^\infty(\mu_{w_n})$. Let $\varDelta \in \ascr$ satisfy $\mu(\varDelta) < \infty$. Set $\varOmega = \phi^{-n}(\varDelta) \in \phi^{-n}(\ascr)$ and $\omega = \phi^{-1}(\varOmega) = \phi^{-(n+1)}(\varDelta)$. Since $\omega \in \phi^{-(n+1)}(\ascr)$, we have
\begin{align*}
    \int_\omega (E_n \circ \phi) \D\mu_{w_{n+1}} &= \int_\omega \esf_{\phi^{n+1}, w_{n+1}}(\hsf_{\phi, w} \circ \phi) \D\mu_{w_{n+1}} = \int_\omega (\hsf_{\phi, w} \circ \phi) \D\mu_{w_{n+1}}.
\end{align*}
The left-hand side of the above is
\begin{align*}
    \int_\omega (E_n \circ \phi) \D\mu_{w_{n+1}} 
    = \int_{\phi^{-1}(\varOmega)} \big(E_n |w_n|^2\big) \circ \phi \D\mu_w
    = \int_\varOmega E_n\,\hsf_{\phi, w} \D\mu_{w_n} .
\end{align*}
The right-hand side is
\begin{align*}
    \int_\omega (\hsf_{\phi, w} \circ \phi) \D\mu_{w_{n+1}} = \int_{\phi^{-1}(\varOmega)} \big(\hsf_{\phi, w} |w_n|^2\big) \circ \phi \D\mu_w 
     = \int_\varOmega \hsf_{\phi, w}^2 \D\mu_{w_n}.
\end{align*}
Hence,
\begin{align*}
    \int_\varOmega E_n \hsf_{\phi, w} \D\mu_{w_n} = \int_\varOmega \hsf_{\phi, w}^2 \D\mu_{w_n}.
\end{align*}
Furthermore, because $E_n \chi_\varOmega$ is $\phi^{-n}(\ascr)$-measurable, $\mu_{w_n}(\varOmega)<\infty$, and $\esf_{\phi^{n}, w_n}$ is an orthogonal projection, we have
\begin{align*}
    \int_\varOmega E_n^2 \D\mu_{w_n} = \int_\varOmega E_n \esf_{\phi^n, w_n}(\hsf_{\phi, w}) \D\mu_{w_n} = \int_\varOmega E_n \hsf_{\phi, w} \D\mu_{w_n} = \int_\varOmega \hsf_{\phi, w}^2 \D\mu_{w_n}.
\end{align*}
Evaluating the $L^2(\mu_{w_n})$ norm we get
\begin{align*}
    \int_\varOmega |\hsf_{\phi, w} - E_n|^2 \D\mu_{w_n} &= \int_\varOmega \hsf_{\phi, w}^2 \D\mu_{w_n} - 2 \int_\varOmega E_n \hsf_{\phi, w} \D\mu_{w_n} + \int_\varOmega E_n^2 \D\mu_{w_n} \\
    &= \int_\varOmega \hsf_{\phi, w}^2 \D\mu_{w_n} - 2 \int_\varOmega \hsf_{\phi, w}^2 \D\mu_{w_n} + \int_\varOmega \hsf_{\phi, w}^2 \D\mu_{w_n} = 0.
\end{align*}
This implies $\esf_{\phi^n, w_n}(\hsf_{\phi, w}) = \hsf_{\phi, w}$ a.e. $[\mu_{w_n}]$ on $\varOmega$. Using $\sigma$-finiteness of $\mu$, we deduce that $\esf_{\phi^n, w_n}(\hsf_{\phi, w}) = \hsf_{\phi, w}$ a.e. $[\mu_{w_n}]$.

(F) $\Rightarrow$ (G) Use \cite[Lemma 5]{2018-lnim-budzynski-jablonski-jung-stochel}.

(H) $\Rightarrow$ (D) Obvious.

(A) $\Rightarrow$ (H) Fix $n\in\nbb$. Clearly, $\cfw^n=C_{\phi^n, w_n}$ is centered. Therefore, using (F), we get that for every $k\in\nbb$, $\hsf_n=\esf_{{nk}}(\hsf_n)$ a.e. $[\mu_{w_{nk}}]$ (use also $(\phi^n)^k=\phi^{nk}$ and $(w_n)_k=w_{nk}$). Since $\hsf_n=\esf_{{j}}(\hsf_n)$ a.e. $[\mu_{w_{j}}]$ implies $\hsf_n=\esf_{{i}}(\hsf_n)$ a.e. $[\mu_{w_{i}}]$ for $i\leqslant j$, we deduce (H).
\end{proof}

\begin{rem}
Condition (F) in Theorem \ref{dan01} admits a spectral interpretation similar to that of condition (D). Applying Theorem \ref{cieczka01} with $\cfw^n$ in place of $\cfw$ and $\hsf_{\phi, w}$ in place of $\hsf_{\phi^n, w_n}$, we see that (F) is equivalent to
\begin{align*}
    \psf_{\phi^n, w_n} M_{\hsf_{\phi, w}} \subseteq M_{\hsf_{\phi, w}} \psf_{\phi^n, w_n}, \quad n \in \nbb.
\end{align*}
Since $\cfw \in \ogr{L^2(\mu)}$, the multiplication operator $M_{\hsf_{\phi, w}}$ is bounded and self-adjoint. Consequently, the above inclusion is equivalent to the commutativity of the orthogonal projection $\psf_{\phi^n, w_n}$ and $M_{\hsf_{\phi, w}}$.
Finite versions of this equivalence will play a central role in the investigation of spectrally $n$-weakly centered unbounded weighted composition operators, which is a subject of a forthcoming paper.
\end{rem}
\begin{rem}
Quasinormal operators are centered, and centered operators are weakly centered. As already seen, the following conditions:
\begin{enumerate}
\item[(i)] $\esf_{\phi}(\hsf_{\phi})=\hsf_{\phi}$
\item[(ii)] $\esf_{\phi}(\hsf_{\phi^n})=\hsf_{\phi^n}$ for every $n\in \nbb$,
\end{enumerate}
are relevant when considering centered and weakly-centered composition operators. The following two conditions, characterizing together quasinormal composition operators (see \cite[Theorem 3.1]{2014-jmaa-budzynski-jablonski-jung-stochel}):
\begin{enumerate}
\item[(iii)] $\esf_{\phi}(\hsf_\phi)=\hsf_\phi\circ \phi$ ,
\item[(iv)] $\esf_{\phi}(\hsf_{\phi^n})=\esf_{\phi}(\hsf_\phi)^n$ for every $n\in\nbb$,
\end{enumerate}
are worth recalling. Observe that (ii) and (iv) together imply $\hsf_{\phi^n}=\hsf_{\phi}^n$ for every $n\in\nbb$, which in  view of \cite[Theorem 3.1]{2014-jmaa-budzynski-jablonski-jung-stochel} implies that $C_\phi$ is quasinormal. In particular, for such a composition operator (iii)  is automatically satisfied.
\end{rem}

\section{Examples}
Morrel and Muhly (\cite{1974-sm-morrel-muhly}) divided centered operators into four distinctive types according to the properties of the phase $U$ in the polar decomposition $T=U|T|$ of the operator in question. Recall that $T$ is {\em type I} (resp. {\em II, III, IV}) {\em centered} if $U$ is a pure isometry (resp. pure co-isometry, orthogonal sum of nilpotent operators, unitary).
Writing it in terms of intersections of ranges of powers of $T$ we have\footnote{Clearly, $\overline{\ob{|T^n|}}$ can be replaced by $\overline{\ob{T^{*n}}}$ in the above classification, but in the context of {\em wco}'s  $|T|$ is easier to handle than $T^*$}
\begin{itemize}
    \item $T$ is type I centered $\iff \bigcap_{n=1}^\infty \overline{\ob{T^n}} = \{0\}$ and $\bigcap_{n=1}^\infty \overline{\ob{|T^n|}} = \mathcal{H}$,
    \item $T$ is type II centered $\iff \bigcap_{n=1}^\infty \overline{\ob{T^n}} = \mathcal{H}$ and $\bigcap_{n=1}^\infty \overline{\ob{|T^n|}} = \{0\}$,
    \item $T$ is type III centered $\iff \bigcap_{n=1}^\infty \overline{\ob{T^n}} = \{0\}$ and $\bigcap_{n=1}^\infty \overline{\ob{|T^n|}} = \{0\}$,
    \item $T$ is type IV centered $\iff \bigcap_{n=1}^\infty \overline{\ob{T^n}} = \mathcal{H}$ and $\bigcap_{n=1}^\infty \overline{\ob{|T^n|}} = \mathcal{H}$.
\end{itemize}
We easily deduce the following.
\begin{pro}\label{wco_types}
Let $\cfw\in\ogr{L^2(\mu)}$ be centered. Then the following conditions are satisfied:
\begin{enumerate}
    \item[(i)] $\cfw$ is of type I or type IV if and only if $\mu\big(X\setminus \bigcap_{n=1}^\infty\{\hsf_{\phi^n, w_n}>0\}\big)=0$.
    \item[(ii)] $\cfw$ is of type II or type III if and only if $\mu\big(X\setminus \bigcup_{n=1}^\infty\{\hsf_{\phi^n, w_n}=0\}\big)=0$.
\end{enumerate}
\end{pro}
\begin{proof}
Use $\jd{|C_{\phi, w}^n|}=\jd{|C_{\phi^n, w_n}|}=\jd{M_{\hsf_{\phi^n, w_n}}}=\chi_{\{\hsf_{\phi^n,w_n}=0\}}L^2(\mu)$, $n\in\nbb$.
\end{proof}
This section offers some insight into types of centered operators in the context of {\em wco}'s . We begin with an example of a family of type IV centered composition operators.
\begin{exa}
Let $\varkappa\in \nbb$. Let $X=\rbb^\varkappa$ and $\ascr=\borel{\rbb^\varkappa}$. Let $\rho(z) = \sum_{k=0}^\infty a_k z^k$, $z \in \cbb$, be an entire function such that $a_n$ is non-negative for every $k\in\zbb_+$ and $a_{k_0} > 0$ for some $k_0 \Ge 1$. Let $\mu=\mu^\rho$ be the $\sigma$-finite measure on $\ascr$ defined by
\begin{align*}
\mu^\rho(\sigma) = \int_\sigma \rho\big(\|x\|^2\big) \D\M_\varkappa(x),\quad \sigma\in\borel{\rbb^\varkappa},
\end{align*}
where $\M_\varkappa$ is the $\varkappa$-dimensional Lebesgue measure on $\rbb^\varkappa$ and $\|\cdot\|$ is a norm on $\rbb^\varkappa$ induced by an inner product. Finally, let $\phi\colon \rbb^\varkappa\to\rbb^\varkappa$ be invertible and linear. Then $\phi$ induces a composition operator $C_\phi$ in $L^2(\mu)$. The boundedness of $C_\phi$ has been fully characterized in \cite[Proposition 2.2]{1990-hmj-stochel}: \emph{if $\rho$ is a polynomial, then $C_\phi$ is bounded; if $\rho$ is not a polynomial, then $C_\phi$ is bounded if and only if  $\|\phi^{-1}\|\Le 1$}. Therefore, if $\rho$ is polynomial and/or $\|\phi^{-1}\|\Le 1$, we have $C_\phi\in \ogr{L^2(\mu)}$.

Since $\phi$ is an invertible transformation of $\rbb^n$, $\esf_\phi$ acts as an identity on $L^2(\mu)$. Thus, by Theorem \ref{dan01}, $C_\phi$ is centered. By the change of the variables theorem we have
\begin{align*} %\label{rn-matrix}
\hsf_{\phi^n}(x)= \frac{1}{|\det \phi^n|} \frac{\rho\big(\|\phi^{-n}(x)\|^2\big)}{\rho\big(\|x\|^2\big)}>0 \quad \text{ for $\mu^\rho$-a.e. } x \in \rbb^n.
\end{align*}
This and Proposition \ref{wco_types} immediately proves that $C_\phi$ is either type I or type IV centered.

Clearly, $f\in L^2(\mu^\rho)$ belongs to $\ob{C_\phi^n}$ if and only if there exists $g \in L^2(\mu^\rho)$ such that $f = g \circ \phi^n$, which implies $g = f \circ \phi^{-n}$. Thus, $f \in \ob{C_\phi^n}$ if and only if $f \circ \phi^{-n} \in L^2(\mu^\rho)$. Let $f \in C_c^\infty(\rbb^\varkappa)$, the space of smooth compactly supported functions on $\rbb^\varkappa$. Then $f \circ \phi^{-n}\in C_c^\infty(\rbb^\varkappa)\subseteq L^2(\mu^\rho)$. Since $C_c^\infty(\rbb^\varkappa)$ is dense in $L^2(\mu)$ for any Radon measure $\mu$, we see that
\begin{align*}
\bigcap_{n=1}^\infty \overline{\ob{C_\phi^n}} = L^2(\mu).    
\end{align*}
This implies that $C_\phi$ has to be type IV centered.
\end{exa}
The above example leads to the following criterion.
\begin{pro}
Let $C_{\phi,w}\in\ogr{L^2(\mu)}$ is injective and has dense range. Then $C_{\phi,w}$ is type IV centered.
\end{pro}
\begin{proof}
Since $C_{\phi,w}$ has dense range, we infer from \cite[Proposition 17(vi)]{2018-lnim-budzynski-jablonski-jung-stochel} that $w \neq 0$ a.e. $[\mu]$. The injectivity of $C_{\phi,w}$ implies, by \cite[Proposition 12]{2018-lnim-budzynski-jablonski-jung-stochel}, that $h_{\phi,w} > 0$ a.e. $[\mu_w]$.

Since $C_{\phi,w}$ has dense range, its adjoint $C_{\phi,w}^*$ is injective. According to \cite[Proposition 17(iii)]{2018-lnim-budzynski-jablonski-jung-stochel}, the kernel of the adjoint is given by $\jd{C_{\phi,w}^*} = \{f \in L^2(\mu) : E_{\phi,w}(f_w) = 0 \text{ a.e. } [\mu_w]\}$. Since $w \neq 0$ a.e. $[\mu]$, the mapping $f \mapsto f_w$ is a bijection from $L^2(\mu)$ onto $L^2(\mu_w)$. Consequently, the condition $\jd{C_{\phi,w}^*} = \{0\}$ implies that $\jd{\efw} = \{0\}$.
By \cite[Lemma 62]{2018-lnim-budzynski-jablonski-jung-stochel}, the triviality of the kernel of $E_{\phi,w}$ is equivalent to $\ob{E_{\phi,w}} = L^2(\mu_w)$. Since $E_{\phi,w}$ is an orthogonal projection in $L^2(\mu_w)$, we conclude that $E_{\phi,w}=I$. In particular, $\cfw$ is centered by Theorem \ref{dan01}

The condition $E_{\phi,w} = I$ implies that for every $n \in \mathbb{N}$,  $E_{\phi^n, w_n}$ acts as the identity operator on $L^2(\mu_{w_n})$. Consequently, by \cite[Proposition 17(iii)]{2018-lnim-budzynski-jablonski-jung-stochel}, $\jd{C_{\phi,w}^{*n}} = \{0\}$, which means that $\ob{\cfw^n}$ is dense in $L^2(\mu)$ for all $n \in \mathbb{N}$. Moreover, the injectivity of $C_{\phi,w}$ implies the injectivity of $C_{\phi,w}^n$, which by \cite[Proposition 12]{2018-lnim-budzynski-jablonski-jung-stochel} yields $h_{\phi^n, w_n} > 0$ a.e. $[\mu]$. Thus, the range of the modulus $|C_{\phi,w}^n| = M_{\sqrt{h_{\phi^n, \hat{w}_n}}}$ is dense in $L^2(\mu)$. Therefore,
\begin{align*}
\bigcap_{n=1}^\infty \overline{\ob{\cfw^n}} = L^2(\mu) \quad \text{and} \quad \bigcap_{n=1}^\infty \overline{\ob{|\cfw^n|}} = L^2(\mu).
\end{align*}
This shows that $C_{\phi,w}$ is a centered operator of type IV.
\end{proof}
\begin{rem}
If $T$ is centered, then obviously $T^*$ is centered. Also, we have:
\begin{enumerate}
    \item $T$ is of type I if and only if $T^*$ is of type II.
    \item $T$ is of type III if and only if $T^*$ is of type III.
    \item $T$ is of type IV if and only if $T^*$ is of type IV.
\end{enumerate}
\end{rem}
Below we give an example of type II centered wco.
\begin{exa}
Let $X = [0, \infty)$, $\ascr=\borel{X}$, and $\mu$ be the restriction of the Lebesgue measure to $\ascr$. Let $\phi: X \to X$ be the affine transformation $\phi(x) = x+1$. Then $C_\phi$ is well defined. In fact, $C_\phi$ is a co-isometry on $L^2(\mu)$. One can easily prove that $\esf_\phi=I$, which implies that $\cfw$ is centered. Let $T = C_\phi^*$. By induction
\begin{align*}
(T^n f)(x) = \begin{cases} 
      f(x-n) & \text{if } x \ge n, \\
      0 & \text{if } 0 \le x < n,
   \end{cases}
   \quad n\in\nbb.
\end{align*}
This implies that $T^n f=\chi_{[n,+\infty)}T^n f$. Consequently, we have
\begin{align}\label{string01}
\bigcap_{n=1}^\infty \overline{\ob{T^n}} \subseteq \bigcap_{n=1}^\infty \big\{ g \in L^2(\mu) : g=0 \text{ on } [0, n) \big\} = \{0\}.    
\end{align}
We now turn to the adjoint of $T$, that is $C_\phi$. We claim that $C_\phi$ is surjective. Indeed, for $g \in L^2(\mu)$ and $n\in\nbb$ we define $f_{g,n} \in L^2(\mu)$ by
\begin{align*}
f_{g,n}(x) = \begin{cases} 
      g(x-n) & \text{if } x \ge n, \\
      0 & \text{if } 0 \le x < n.
   \end{cases}
\end{align*}
It is easy to see that $C+\phi^n f_{g,n}=$. Hence, $\ob{T^{*n}} = L^2(\mu)$ for all $n \in \mathbb{N}$. Therefore,
\begin{align*}
\bigcap_{n=1}^\infty \overline{\ob{T^{*n}}} = L^2(\mu).    
\end{align*}
This and \eqref{string01} imply that $T$ is a type I centered and $C_\phi$ is type II centered.
\end{exa}

We now turn our attention towards weighted shifts on directed trees.  We begin by introducing necessary notation. 

Let $\tcal=(V,E)$ be a directed tree ($V$ and $E$ stand for the sets of vertexes and edges of $\tcal$, respectively). Set $\dzi u = \{v\in V\colon (u,v)\in E\}$ for $u \in V$. Denote by $\paa$ the partial function from $V$ to $V$ which assigns to a vertex $u\in V$ its parent $\pa{u}$ (i.e.\ a unique $v \in V$ such that $(v,u)\in E$). For $v\in V$, we define inductively $\dzin{n+1}{v}=\dzi{\dzin{n}{v}}$, $n\in\nbb$. A vertex $u \in V$ is called a {\em root} of $\tcal$ if $u$ has no parent. A root is unique (provided it exists); we denote it by $\koo$. Set $V^\circ=V \setminus \{\koo\}$ if $\tcal$ has a root and $V^\circ=V$ otherwise. We say that $u \in V$ is a {\em branching vertex} of $V$, and write $u \in V_{\prec}$, if $\dzi{u}$ consists of at least two vertexes. 

Assume $\lambdab=\{\lambda_v\}_{v \in V^{\circ}} \subseteq \cbb$ satisfies $\sup_{v\in V}\sum_{u\in\dzii{v}}|\lambda_u|^2<\infty$. Then the following formula
\begin{align*}%\label{xx}
(\slam f)(v)=
   \begin{cases}
\lambda_v \cdot f\big(\pa v\big) & \text{ if } v\in V^\circ,
   \\
0 & \text{ if } v=\koo,
   \end{cases}
\end{align*}
defines a bounded operator on $\ell^2(V)$ (as usual, $\ell^2(V)$ is the Hilbert space of square summable complex functions on $V$ with standard inner product). We call it a {\em weighted shift on a directed tree $\tcal$ with weights $\lambdab$}. We refer the interested reader to an excellent monograph \cite{2012-mams-j-j-s} on these operators. We will abbreviate ``weighted shift on a directed tree'' (resp. ``weighted shift'') to {\em wsdt} (resp. {\em ws}).

A {\em ws} $\slam$ on $\tcal=(V,E)$ with $\card{V}\leqslant\aleph_0$ can be viewed as a {\em wco} $\cfw$ in $L^2(X, \ascr, \mu)$, where $X=V$, $\ascr=2^V$, $\mu$ is the counting measure on $V$, $\phi$ is any extension of the $\paa$ partial function, and $w$ is given by
\begin{align*}
w(x)=
   \begin{cases}
\lambda_x & \text{ if } x\in V^\circ,
   \\
0 & \text{ if } v=\koo\text{ (in case it exists)}.
   \end{cases}
\end{align*}

All classical {\em ws}'s  are centered. Contrary to that, {\em wsdt}'s  may be not centered. In fact, even the simplest tree not related to classical shifts (i.e., a rootless directed tree with one branching vertex of valency 2) admits a non centered {\em ws}. This was shown in \cite[Example 12]{2025-arxiv-budzynski}. Centered {\em wsdt}'s  have a neat characterization - it essentially says that for centered $\slam$ the sums $\sum_{u\in \dzi{v}}|\lambda_u|^2$ are constant across every generation.
\begin{pro}\label{burasuka01}
Let $\slam\in\ogr{\ell^2(V)}$ with $\card{V}\leqslant\aleph_0$. Then $\slam$ is centered if and only if for every $u_1, u_2\in \dzi{v}$ such that there is $n\in\zbb_+$ such that $\paa^{n+1}(u_1)=\paa^{n+1}(u_2)$ and 
\begin{align*}
\lambda_{u_1}\lambda_{\paa(u_1)}\cdots \lambda_{\paa^n(u_1)}\lambda_{u_2}\lambda_{\paa(u_2)}\cdots\lambda_{\paa^n(u_2)}\neq 0
\end{align*}
the equality
\begin{align*}
\sum_{y\in\dzi{u_1}}|\lambda_{y}|^2=\sum_{y\in\dzi{u_2}}|\lambda_{y}|^2
\end{align*}
holds.
\end{pro}
\begin{proof}
Use the equality (see \cite[Proposition 79]{2018-lnim-budzynski-jablonski-jung-stochel})
\begin{align*}%\label{mucha1}
\hfw(x)={\sum_{y\in\phi^{-1}(\{x\})}|w(y)|^2}=\sum_{y\in\dzi{x}}|\lambda_y|^2,\quad x\in X,
\end{align*}
the fact that $\phi^{-n}(\ascr)$-measurable functions are constant on sets $\phi^{-n}(v)$, $v\in V$, and apply Theorem \ref{dan01}\,(F).
\end{proof}
It is easy to see that ``$\paa^{k}(u)=\paa^{k}(v)$ for some $k\in\nbb$'' defines an equivalence relation $\sim$ on $V$. The equivalence classes $[v]_{\sim}$, $v\in V$, form {\em generations} of $\tcal$. Proposition \ref{burasuka01} implies that if $\lambda_v\neq0$ for all $v\in V$, $\slam$ is centered if and only if 
\begin{align*}
u\mapsto \sum_{y\in\dzi{u}}|\lambda_y|^2
\end{align*}
is constant on $[v]_\sim$, $v\in V^\circ$. In particular, the trees supporting centered unweighted shifts have to be ''generationally regular''.
\begin{cor}
Let $\tcal=(V,E)$ be a directed tree. Let $S \in \ogr{\ell^2(V)}$ be a weighted shift on $\tcal$ such that $\lambda_v = 1$ for all $v \in V$). Then $S$ is centered if and only if the valency $\card{\dzi{u}}$ is constant on every generation $[v]_\sim$, $v\in \tcal$.
\end{cor}
\begin{proof}
Since $\lambda_v=1$ for all $v \in V$, we have $\sum_{y \in \dzi{u}} |\lambda_y|^2 = \sum_{y \in \dzi{u}} 1 = \card{\dzi{u}}$.
By Proposition \ref{burasuka01}, $S$ is centered if and only if this sum is constant on every generation $[v]_\sim$. Thus, $S$ is centered if and only if for every generation, all vertices in that generation have the same number of children.
\end{proof}

Below we use (2) above and provide a simple example of a {\em wsdt}'s  which is weakly centered but not centered.
\begin{exa}\label{blackblack+}
Let $\tcal = (V,E)$ be the directed tree with $$V=\{-k\colon k\in\zbb_+\}\cup\big\{(1,1),(1,2),(2,1),(2,2)\big\}\cup\big\{(n,i)\colon n\in\nbb, n\geqslant, i\in\{1,2,3\} \big\}$$ and $E$ as in as in Figure \ref{truffaz-fig}.\allowdisplaybreaks
\begin{figure}[ht]
\begin{center}
\begin{tikzpicture}[scale=0.8, transform shape,edge from parent/.style={draw,to}]
\tikzstyle{every node} = [circle,fill=gray!30]
\node[fill=none] (e-3) at (-4,0) {};
\node (e-1)[font=\footnotesize, inner sep = 3pt] at (-2,0) {$-1$};
\node (e10)[font=\footnotesize, inner sep = 3pt] at (0,0) {$0$};

\node (e11)[font=\footnotesize, inner sep = 1pt] at (4.5,1.5) {$(2,1)$};
\node (e12)[font=\footnotesize, inner sep = 1pt] at (7,1.5) {$(3,1)$};
\node[fill = none] (e1n) at(9,1.5) {};

\node (g11)[font=\footnotesize, inner sep = 1pt] at (2,1.5) {$(1,1)$};

\node (f11)[font=\footnotesize, inner sep = 1pt] at (2,-1.5) {$(1,2)$};
\node (f12)[font=\footnotesize, inner sep = 1pt] at (4.5,-1.5) {$(2,2)$};
\node (f14)[font=\footnotesize, inner sep = 1pt] at (7,-3) {$(3,3)$};
\node[fill = none] (f1n) at(9,-3) {};
\node[fill = none] (f2n) at(9,0) {};

\node (f13)[font=\footnotesize, inner sep = 1pt] at (7,0) {$(3,2)$};

\draw[dashed,->] (e-3) --(e-1) node[pos=0.5,above = 0pt,fill=none] {$1$};
\draw[->] (e-1) --(e10) node[pos=0.5,above = 0pt,fill=none] {$1$};
\draw[->=stealth] (g11) --(e11) node[pos=0.5,above = 0pt,fill=none] {$1$};
\draw[->=stealth] (e10) --(g11) node[pos=0.5,above = 0pt,fill=none] {$1$};

\draw[->] (e10) --(f11) node[pos=0.5,above = 0pt,fill=none] {$1$};
\draw[->] (e11) --(e12) node[pos=0.5,above = 0pt,fill=none] {$1$};
\draw[->] (f12) --(f13) node[pos=0.5,above = 0pt,fill=none] {$1$};
\draw[->] (f11) --(f12) node[pos=0.5,below = 0pt,fill=none] {$1$};
\draw[->] (f12) --(f14) node[pos=0.5,above = 0pt,fill=none] {$1$};
\draw[dashed, ->] (f14)--(f1n);
\draw[dashed, ->] (f13)--(f2n);
\draw[dashed, ->] (e12)--(e1n);
\end{tikzpicture}
\end{center}
\caption{\label{truffaz-fig} The directed tree $\tcal$ considered in Example \ref{blackblack+}.}
\end{figure} 
Let $\lambdab=\{\lambda_v\}_{v\in{V^\circ}}$ satisfy  $\lambda_v=1$ for $v\in V$. Clearly, in view of (2), $\slam\in\ogr{\ell^2(V)}$ is not centered. However, in view of Proposition \cite[Proposition 11]{2025-arxiv-budzynski} weakly centered. 
\end{exa}
\begin{rem}
In view of \cite[Proposition 3.1.6]{2012-mams-j-j-s}, any {\em ws} $\slam\in\ogr{\ell^2(V)}$ can be decomposed into an orthogonal sum $\bigoplus_{j\in J}S_{\lambdab^j}$ of {\em ws}'s  $S_{\lambdab^j}$ with nonzero weights. Informally speaking, $\slam$ is cut into pieces at edges ending at vertexes with zero weights. One can use this fact and Proposition \ref{burasuka01} to show that $\slam$ is centered if and only if each $S_{\lambdab^j}$ in the decomposition $\bigoplus_{j\in J}S_{\lambdab^j}$ is centered. This implies that when dealing with centered weighted shifts on directed trees we may in many cases restrict ourselves to the ones that have nonzero weights. 
\end{rem}
As shown by Morrell and Mulhy, every centered $T\in\ogr{\hh}$ can be decomposed into 
\begin{align*}
T=T|_{\hh_I}\oplus T|_{\hh_{II}}\oplus T|_{\hh_{III}}\oplus T|_{\hh_{IV}},
\end{align*}
with $T|_{\hh_{N}}$ being type $N$ centered, $N\in\{I, II, III, IV\}$. For injective centered operators type II and type III summands are trivial.
\begin{lem}\label{red}
Let $T \in \boldsymbol{B}(\mathcal{H})$ be a centered operator. If $T$ is injective, then $\hh_{II}$ and $\hh_{III}$ are trivial. Consequently, $T=T|_{\hh_I}\oplus T|_{\hh_{IV}}$.
\end{lem}
\begin{proof}
Since $T$ is injective, $\jd{T^n} = \{0\}$ for all $n \in \mathbb{N}$. Hence, we obtain
\begin{align}\label{red01}
\bigcap_{n=1}^\infty\overline{\ob{|T^n|}} = \bigcap_{n=1}^\infty \jd{T^n}^\perp = \mathcal{H}.
\end{align}
Let $\mathcal{H} = \mathcal{H}_I \oplus \mathcal{H}_{II} \oplus \mathcal{H}_{III} \oplus \mathcal{H}_{IV}$ be the decomposition of $\mathcal{H}$ into reducing subspaces corresponding to the type I, type II, type III, and type IV parts of $T$, respectively. Since each of these subspaces reduces $T$, it also reduces $|T^n|$ for all $n \in \mathbb{N}$. Therefore, we get
\begin{align}\label{red02}
\bigcap_{n=1}^\infty \overline{\ob{|T^n|}} 
= \left( \bigcap_{n=1}^\infty \overline{\ob{|T|_{\hh_{I}}^n}} \right) \oplus \dots \oplus \left( \bigcap_{n=1}^\infty \overline{\ob{|T|_{\hh_{IV}}^n}} \right).
\end{align}
Type II and type III parts of $T$ must satisfy
\begin{align*}
\bigcap_{n=1}^\infty \overline{\ob{|T|_{\hh_{N}}^n}} = \{0\}, \quad N \in \{II, III\}.
\end{align*}
This, \eqref{red01}, and \eqref{red02} imply that 
\begin{align*}
\hh_{N}=P_{\hh_N} \hh= P_{\hh_N} \bigcap_{n=1}^\infty \overline{\ob{|T|_{\hh_{N}}^n}} = \{0\},\quad N \in \{II, III\}.
\end{align*}
Consequently, $T$ is the orthogonal sum of a type I centered operator and a type IV centered operator.
\end{proof}
Using the above we deduce a criterion for a weighted shifts on leafless and rootless directed trees to be type I centered .
\begin{cor}\label{cook01}
Let $\tcal=(V,E)$ be a rootless and leafless directed tree. Let $S_\lambda \in \ogr{\ell^2(v)}$ be a weighted shift on $\tcal$ with nonzero\footnote{With nonzero weights, $V$ is required to be at most countable whenever $\slam$ is bounded} weights $\lambda=\{\lambda_v\}_{v \in V}$.
If $S_\lambda$ is centered, then the following conditions are equivalent:
\begin{enumerate}
    \item[(i)] $S_\lambda$ is of type I centered,
    \item[(ii)] $\bigcap_{n=1}^\infty \overline{\ob{\slam^n}} = \{0\}$.
\end{enumerate}
Moreover, if $\slam$ is type I centered, then the following condition hold:
\begin{enumerate}
    \item[(iii)] $\tcal$ is not isomorphic to $\mathbb{Z}$ (equivalently, $\tcal$ has at least one branching vertex)
\end{enumerate}
\end{cor}
\begin{proof}
In view of Lemma \ref{red}, type II and type III centered parts of $\slam$ ar trivial.

(i) $\Leftrightarrow$ (ii)
$S_\lambda$ being type I centered means $\hh_{IV}=\{0\}$ and $\bigcap_{n=1}^\infty \overline{\ob{\slam|_{\hh_I}^n}}=\{0\}$. Since
\begin{align*}
\bigcap_{n=1}^\infty \overline{\ob{\slam^n}} = \bigcap_{n=1}^\infty \overline{\ob{\slam|_{\hh_I}^n}} \oplus \bigcap_{n=1}^\infty \overline{\ob{\slam|_{\hh_{IV}}^n}},
\end{align*}
we deduce the equivalence of (i) and (ii).

(ii) $\Rightarrow$ (iii) Suppose, contrary to our claim, that $\tcal$ is isomorphic to $\mathbb{Z}$. Then, by \cite[Remark 3.1.4]{2012-mams-j-j-s}, $S_\lambda$ is unitarily equivalent to a bilateral weighted shift with nonzero weights. Clearly, for such a $\slam$ we have $\overline{\ob{\slam^n}} = l^2(V)$ for all $n$. This contradicts (ii).
\end{proof}

\begin{exa}\label{exa:binary_type1}
Let $\tcal=(V,E)$ be a rootless and leafless directed tree such that $\card{\dzi{u}}=2$ for all $u\in V$. Let $\slam \in \ogr{\ell^2(V)}$ be a weighted shift with weights $\lambda_v=1$ for all $v\in V$. We show that  $\slam$ is type I centered.

We first observe that for any $u \in V$, $\sum_{y\in\dzi{u}}|\lambda_y|^2 = 1^2 + 1^2 = 2$, i.e., it is constant on every generation $[v]_\sim$. Hence, $\slam$ is a centered operator by Proposition \ref{burasuka01}. In view of Corollary \ref{cook01}, to prove that $\slam$ is of type I it suffices to show that $\bigcap_{n=1}^\infty \ob{\slam^n}=\{0\}$. Suppose $f \in \ob{\slam^n}$ for some $n\in\nbb$. By \cite[Lemma 3.1.2]{2012-mams-j-j-s}, we have
\begin{align*}
\slam^n e_u = \sum_{v \in \dzin{n}{u}} e_v,\quad u\in V, n\in \nbb.
\end{align*}
Thus $f = \slam^n g =\sum_{u \in V} g(u) \sum_{v \in \dzin{n}{u}} e_v$ with some $g \in \ell^2(V)$. Because descendants $\dzin{n}{u}$ are pairwise disjoint sets for distinct $u$ in the same generation, $f$ must be constant on $\dzin{n}{u}$ for every $u \in V$. Specifically, if $v \in \dzin{n}{u}$, then $f(v) = g(u)$.
Consequently,
\begin{align*}
\|f\|^2 = \sum_{v \in V} |f(v)|^2 \Ge \sum_{v \in \dzin{n}{u}} |f(v)|^2 = \sum_{v \in \dzin{n}{u}} |g(u)|^2 = 2^n |g(u)|^2 = 2^n |f(v)|^2,
\end{align*}
where $v$ is any fixed vertex in $\dzin{n}{u}$. This implies $|f(v)|^2 \Le 2^{-n} \|f\|^2$ for such a $v$. If $f \in \bigcap_{n=1}^\infty \ob{\slam^n}$, we get $f(v)=0$ for all $v \in V$. Thus, $\bigcap_{n=1}^\infty \ob{\slam^n}$ is trivial.
\end{exa}
\begin{exa}\label{exa:Y_tree_not_type1}
Consider the directed tree $\tcal=(V,E)$ defined as follows (see Figure \ref{fig:Y_tree}):
\begin{itemize}
    \item $V = \{ -k \colon k \in \nbb_0 \} \cup \{ (i,j) \colon i \in \{1,2\}, j \in \nbb \}$,
    \item $E = \{ (-k-1, -k) \colon k \in \nbb_0 \} \cup \{ (0, (1,1)), (0, (2,1)) \} \cup \{ ((i,j), (i,j+1)) \colon i \in \{1,2\}, j \in \nbb \}$.
\end{itemize}
Let $\slam$ be the weighted shift on $\tcal$ with weights $\lambda_v=1$ for all $v \in V$. Note that $\card{\dzi{0}}=2$ and $\card{\dzi{v}}=1$ for all $v \neq 0$.
\begin{figure}[ht]
\begin{center}
\begin{tikzpicture}[scale=0.8, transform shape]
\tikzstyle{every node} = [circle,fill=gray!30]

% Trunk nodes
\node[fill=none] (trunk_inf) at (-5.5,0) {};
\node (trunk_2) [font=\footnotesize, inner sep=3pt] at (-3.5,0) {$-2$};
\node (trunk_1) [font=\footnotesize, inner sep=3pt] at (-1.5,0) {$-1$};
\node (root) [font=\footnotesize, inner sep=3pt] at (0.5,0) {$0$};

% Branch 1 (Upper) nodes
\node (b1_1) [font=\footnotesize, inner sep=1pt] at (2.5,1.5) {$(1,1)$};
\node (b1_2) [font=\footnotesize, inner sep=1pt] at (4.5,1.5) {$(1,2)$};
\node[fill=none] (b1_inf) at (6,1.5) {};

% Branch 2 (Lower) nodes
\node (b2_1) [font=\footnotesize, inner sep=1pt] at (2.5,-1.5) {$(2,1)$};
\node (b2_2) [font=\footnotesize, inner sep=1pt] at (4.5,-1.5) {$(2,2)$};
\node[fill=none] (b2_inf) at (6,-1.5) {};

% Edges (Arrows indicate parent -> child direction)
% Trunk edges
\draw[dashed,->] (trunk_inf) -- (trunk_2);
\draw[->] (trunk_2) -- (trunk_1)node[pos=0.5,above = 0pt,fill=none]{$1$};
\draw[->] (trunk_1) -- (root)node[pos=0.5,above = 0pt,fill=none]{$1$};

% Branch 1 edges
\draw[->] (root) -- (b1_1)node[pos=0.5,above = 0pt,fill=none]{$1$};
\draw[->] (b1_1) -- (b1_2)node[pos=0.5,above = 0pt,fill=none]{$1$};
\draw[dashed,->] (b1_2) -- (b1_inf);

% Branch 2 edges
\draw[->] (root) -- (b2_1) node[pos=0.5,above = 0pt,fill=none]{$1$};
\draw[->] (b2_1) -- (b2_2) node[pos=0.5,above = 0pt,fill=none]{$1$};
\draw[dashed,->] (b2_2) -- (b2_inf);

\end{tikzpicture}
\end{center}
\caption{\label{fig:Y_tree} The directed tree $\tcal$ considered in Example \ref{exa:Y_tree_not_type1}.}
\end{figure}
$\slam$ is easily seen to be centered by Proposition \ref{burasuka01}. However, $\slam$ is not of type I. For this it suffices to show that $\bigcap_{n=1}^\infty \ob{\slam^n} \neq \{0\}$.

As in Example \ref{exa:binary_type1}, $f \in \ob{\slam^n}$ if and only if $f$ is constant on $\dzin{n}{u}$ for any $u \in V$. In our case, for any $k \in \nbb$, the vertexes $(1,k)$ and $(2,k)$ are the only elements of $\dzin{k}{0}$ (and $\dzin{k+m}{-m}$ for $m\geqslant 1$). Therefore, any function $f \in \ell^2(V)$ satisfying the symmetry condition $f((1,k)) = f((2,k))$ for all $k \in \nbb$ belongs to $\bigcap_{n=1}^\infty \ob{\slam^n}$. In particular, $f \in \ell^2(V)$ defined by
\begin{align*}
f(v) = \begin{cases}
2^{-k} & \text{if } v = (1,k) \text{ or } v = (2,k), \\
0 & \text{otherwise}.
\end{cases}
\end{align*}
satisfies the symmetry condition. This means that the type IV centered summand in the decomposition $\slam=\slam|_{\hh_I}\oplus \slam|_{\hh_{II}}$ is not trivial.
\end{exa}
The argument used in Example \ref{exa:binary_type1} leads to a concrete criterion for $\slam$ to be type I centered.
\begin{pro}\label{bigcook01}
Let $\tcal=(V,E)$ be a rootless and leafless directed tree. Let $S_\lambda \in \ogr{\ell^2(v)}$ be a centered weighted shift on $\tcal$ with nonzero weights $\lambda=\{\lambda_v\}_{v \in V^{\circ}}$. If the following condition hold\footnote{Here, and later on, $\paa^n$ stands for the $n$-fold composition of the $\paa$ partial function while $\lambda_{\paa^n(v)|v}= \lambda_v\cdot \lambda_{\pa{v}}\cdots \lambda_{\paa^{n-1}(v)}$, $v\in V$.}
\begin{align}\label{bigcook02}
\lim_{n \to \infty} \frac{|\lambda_{\paa^n(v)|v}|^2}{\|S_\lambda^n e_{\paa^n(v)}\|^2} = 0,\quad v\in V,
\end{align}
then $S_\lambda$ is type I centered.
\end{pro}

\begin{proof}
By Corollary \ref{cook01}, $\slam =\slam|_{\hh_{I}}\oplus\slam|_{\hh_{IV}}$, with $\slam|_{\hh_{N}}$ being type $N$ centered operator, $N\in\{I, II\}$. We show that $\mathcal{H}_{IV} = \bigcap_{n=1}^\infty \overline{\ob{\slam^n}}$ is trivial. 

Let $f \in \mathcal{H}_{IV}$. Since $S_\lambda|_{\hh_{IV}}$ is unitary, we have $\ob{\slam|_{\hh_{IV}}^n} = \mathcal{H}_{IV}$ for all $n \in \mathbb{N}$. Therefore, for each $n \in \mathbb{N}$, there exists $g_n \in \ell^2(V)$ such that $f = S_\lambda^n g_n$. Fix $v \in V$ and $n \in \mathbb{N}$. Let $u_n = \paa^n(v)$. By \cite[Lemma 6.1.1 (i)]{2012-mams-j-j-s} we have $f(w) = (S_\lambda^n g_n)(w) = \lambda_{u_n|w} g_n(u_n)$ for every $w\in \dzin{n}{u_n}$. Thus
\begin{align*}
|g_n(u_n)|^2 = \frac{|f(v)|^2}{|\lambda_{u_n|v}|^2}.
\end{align*}
Using \cite[Lemma 6.1.1 (ii)]{2012-mams-j-j-s}, we obtain
\begin{align*}
|f(v)|^2 \frac{\|S_\lambda^n e_{u_n}\|^2}{|\lambda_{u_n|v}|^2} & = |g_n(u_n)|^2 \|S_\lambda^n e_{u_n}\|^2= \sum_{w \in \dzin{n}{u_n}} |\lambda_{u_n|w}|^2 |g_n(u_n)|^2\\
&= \sum_{w \in \dzin{n}{u_n}} |f(w)|^2 \leqslant \|f\|^2.
\end{align*}
This yields
\begin{align*}
|f(v)|^2 \le \|f\|^2 \frac{|\lambda_{\paa^n(v)|v}|^2}{\|S_\lambda^n e_{\paa^n(v)}\|^2},\quad v\in V.
\end{align*}
Using \eqref{bigcook02} we get $f=0$, which proves that $\mathcal{H}_{IV} = \{0\}$. Consequently, $S_\lambda$ is type I centered.
\end{proof}
If $\tcal$ contains $\kappa$-ary subtree we get yet another criterion. 
\begin{cor}
Let $\tcal=(V,E)$ be a rootless directed tree such that $\card{\dzi{v}} \ge \kappa$ for all $v \in V$, where $\kappa \geqslant 2$ is an integer. Let $S_\lambda \in \boldsymbol{B}(\ell^2(V))$ be a centered weighted shift on $\tcal$. If
\begin{align}\label{kanary01}
\sup \{|\lambda_v|\colon v \in V\}< \sqrt{\kappa} \, \inf \{|\lambda_v|\colon v \in V\}.
\end{align}
then $S_\lambda$ is type I centered.
\end{cor}
\begin{proof}
Clearly, $\tcal$ is leafless. Also, \eqref{kanary01} implies that all weights are nonzero (and thus $\tcal$ is at most countable). Thus, we can apply Proposition \ref{bigcook01}. 

Let $M=\sup \{|\lambda_v|\colon v \in V\}$ and $m=\inf \{|\lambda_v|\colon v \in V\}$. Fix $v \in V$ and $n \in \mathbb{N}$. We have
\begin{align}\label{kanary02}
|\lambda_{\paa^n(v)|v}|^2 = \prod_{j=0}^{n-1} |\lambda_{\paa^j(v)}|^2 \leqslant M^{2n}.
\end{align}
Since $\card{\dzi{u}} \geqslant \kappa$ for all $u \in V$, we deduce that $\card{\dzin{n}{\paa^n{v}}}\geqslant \kappa^n$. On the other hand, by \cite[Lemma 6.1.1 (ii)]{2012-mams-j-j-s}, we have
\begin{align}\label{kanary03}
\|S_\lambda^n e_{\paa^n(v)}\|^2 = \sum_{w \in \dzin{n}{\paa^n(v)}} |\lambda_{par^n(v)|w}|^2 \geqslant \kappa^n m^{2n}.
\end{align}
Combining \eqref{kanary02} and \eqref{kanary03} we obtain
\begin{align*}
0 < \frac{|\lambda_{\paa^n(v)|v}|^2}{\|S_\lambda^n e_{\paa^n(v)}\|^2} \leqslant \frac{M^{2n}}{\kappa^n m^{2n}} = \left( \frac{M^2}{\kappa m^2} \right)^n.
\end{align*}
Clearly, by \eqref{kanary01}, this implies 
\begin{align*}
\lim_{n \to \infty} \frac{|\lambda_{\paa^n(v)|v}|^2}{\|S_\lambda^n e_{\paa^n(v)}\|^2} = 0.
\end{align*}
Since $v \in V$ was chosen arbitrarily, Proposition \ref{bigcook01} guarantees that $S_\lambda$ is type I centered.
\end{proof}
Example \ref{exa:binary_type1} shows that for the unweighted shift on the full binary tree (which is type I centered), the ratio in \eqref{bigcook02} is equal to $(1/2)^n$. Thus, the limit is $0$. It is natural to ask if condition \eqref{bigcook02} is necessary for $\slam$ to be type I centered.
\begin{opq}
Does there exist a rootless and leafless directed tree $\tcal$ and a weighted shift $\slam$ (with nonzero weights) on $\tcal$ such that $\slam$ is type I centered but
\begin{align*}
\limsup_{n \to \infty} \frac{|\lambda_{\paa^n(v)|v}|^2}{\|S_\lambda^n e_{\paa^n(v)}\|^2} > 0
\end{align*}
for some $v \in V$?
\end{opq}
Another property that excludes a possibility of $\slam$ being type IV is $\tcal$ being rooted and leafless.
\begin{pro}\label{root}
Let $\tcal=(V,E)$ be a rooted and leafless directed tree. Let $\slam\in\ogr{\ell^2(V)}$ be a centered weighted shift on $\tcal$ with nonzero weights $\lambdab=\{\lambda_v\}_{v \in V^{\circ}}$. Then $S_\lambda$ is type I centered.
\end{pro}

\begin{proof}
Since $\tcal$ is leafless and the weights are nonzero, it follows from \cite[Proposition 3.1.7]{2012-mams-j-j-s} that $S_\lambda$ is injective. By Lemma \ref{red}, $S_\lambda=\slam|_{\hh_I}\oplus \slam|_{\hh_{IV}}$, so we need to show that $\bigcap_{n=1}^\infty \overline{\ob{\slam^n}} = \{0\}$.

Since $\tcal$ has a root, by \cite[Corollary 2.1.5]{2012-mams-j-j-s}, $V = \bigsqcup_{k=0}^\infty \dzin{k}{\koo}$. For $u\in V$, let $\text{gen}(u)$ denote a (unique) $k\in\zbb_+$ such that $u \in \dzin{k}{\koo}$. It is easy to see that for $f \in \ell^2(V)$ and $n \in \mathbb{N}$ we have $\big(S_\lambda^n f\big)(v) = 0$ for all $v \in V$ such that $\text{gen}(v) < n$. Hence, if $g \in \overline{\ob{\slam^n}}$, then $g(v)=0$ for every $v\in \bigcup_{k=0}^{n-1} \dzin{k}{\koo}$. Since $V=\bigcup_{k=0}^{\infty} \dzin{k}{\koo}$, we deduce that $\bigcap_{n=1}^\infty \overline{\ob{\slam^n}}=\{0\}$. Thus, the type IV centered component is trivial. Hence, $S_\lambda$ is type I centered.
\end{proof}
\begin{pro}
Let $\tcal=(V,E)$ be a rooted directed tree. Let $\slam \in \ogr{\ell^2(V)}$ be a centered weighted shift on $\tcal$ with nonzero weights $\lambdab=\{\lambda_v\}_{v \in V^{\circ}}$. If $\tcal$ has at least one leaf, then the tree $\tcal$ has finite depth and $S_\lambda$ is type III centered.
\end{pro}
\begin{proof}
The same argument as in the proof of Proposition \ref{root} proves that
$\bigcap_{n=1}^\infty \overline{\ob{\slam^n}} = \{0\}$. This implies that $\mathcal{H}_{II}$ and $\mathcal{H}_{IV}$ are trivial. Thus, $\slam=\slam|_{\hh_{I}}\oplus \slam|_{\hh_{III}}$.

Let $u \in V$ be a leaf. Then $\dzi{u} = \varnothing$, which implies $S_\lambda e_u = 0$. Since $\slam|_{\hh_I}$ is injective, $e_u \in \mathcal{H}_{III}$. In particular, $\{e_v\colon v \text{ is a leaf}\}\subseteq \hh_{III}$. We may assume that $u\neq \koo$. Since $S_\lambda^* e_u = \bar{\lambda}_u e_{\pa{u}}$, $\lambda_u\neq 0$ and $\hh_{III}$ is reducing for $S_\lambda^*$, we get $e_{\pa{u}} \in \hh_{III}$. By induction, $e_{root}\in \hh_{III}$.

We now suppose, contrary to our claim, that $\tcal$ contains an infinite branch, i.e., there is $\{v_n\}_{n\in\zbb_+}\in V$ such that $v_0=\koo$ and $v_{k+1}\in\dzi{v_k}$ for $k\in\zbb_+$. Then
\begin{align}\label{3000}
\|\slam^n e_{\koo}\|^2=\sum_{u\in \dzin{n}{\koo}}|\lambda_{\koo| u}|^2\geqslant |\lambda_{\koo|v_n}|^2>0, \quad n\in \nbb.
\end{align}
On the other hand, since $e_{root}\in \hh_{III}$, we have
\begin{align}\label{3001}
|\slam|_{\hh_{III}}|^n e_{\koo} = |\slam^n| e_{\koo} = \|\slam^n e_{\koo}\| \,e_{\koo},\quad n\in \nbb.
\end{align}
In view of \eqref{3000}, \eqref{3001} implies that $\bigcap_{n=1}^\infty \overline{\ob{|\slam|_{\hh_{III}}}|^n} \neq \{0\}$, a contradiction. Therefore, $\tcal$ has no infinite branches. Consequently, $\{e_v\colon v\in V\}\subseteq \hh_{III}$, which yields $\slam=\slam|_{\hh_{III}}$.
\end{proof}
\begin{rem}
The property of being centered imposes a structural rigidity on the tree $\tcal$ if leaves are present. It follows from Proposition \ref{burasuka01} that the map $u \mapsto \sum_{v \in \dzi{u}} |\lambda_v|^2$ is constant on every generation of $\tcal$. If $\tcal$ has a leaf $u$ in the $n$-th generation (i.e., $u \in \dzin{n}{\koo}$), then $\sum_{v \in \dzi{u}} |\lambda_v|^2 = 0$. Consequently, for every $w \in \dzin{n}{\koo}$, we must have $\sum_{v \in \dzi{w}} |\lambda_v|^2 = 0$, which implies $\dzi{w} = \varnothing$. Thus, if a $\slam$ is centered and the tree $\tcal$ has a leaf, all vertexes in the same generation are leaves. Since $\tcal$ is connected, this implies that $\tcal$ has a uniform finite depth and all maximal branches have the same length.

The property that all leaves belong to the same generation holds for any (connected) directed tree, rooted or not. Indeed, any two vertexes $u, v \in V$ share a common ancestor $w \in V$ (i.e., $w = \paa^n(u) = \paa^m(v)$ for some $n,m \in \zbb_+$). For a centered $\slam$, the condition $\sum_{v \in \dzi{u}} |\lambda_v|^2 = 0$ (which characterizes leaves) propagates across all the vertexes having common ancestor.
\end{rem}

\begin{exa}\label{exa:Z_minus_type2}
Let $\tcal=(V,E)$ be the directed tree with $V = \{-k \colon k \in \zbb_+\}$ and $E = \{(-k-1, -k) \colon k \in \zbb_+\}$. Let $\slam$ be the weighted shift on $\tcal$ with weights $\lambda_v = 1$ for all $v \in V$. Note that this tree is rootless but has a leaf at $v=0$.
\begin{figure}[ht]
\begin{center}
\begin{tikzpicture}[scale=0.8, transform shape]
\tikzset{every node/.style={circle,fill=gray!30}}

% Nodes
\node[fill=none] (inf) at (-8,0) {};
\node (n3) [font=\footnotesize, inner sep=3pt] at (-6,0) {$-3$};
\node (n2) [font=\footnotesize, inner sep=3pt] at (-4,0) {$-2$};
\node (n1) [font=\footnotesize, inner sep=3pt] at (-2,0) {$-1$};
\node (n0) [font=\footnotesize, inner sep=3pt] at (0,0) {$0$};

% Edges (Parent -> Child)
% Parent of -k is -k-1. Arrow goes from -k-1 to -k.
\draw[dashed,->] (inf) -- (n3);
\draw[->] (n3) -- (n2);
\draw[->] (n2) -- (n1);
\draw[->] (n1) -- (n0);

\end{tikzpicture}
\end{center}
\caption{\label{fig:Z_minus_tree} The directed tree $\tcal$ isomorphic to $\zbb_-$, considered in Example \ref{exa:Z_minus_type2}.}
\end{figure}

We verify the conditions for $\slam$ to be type II centered:
\begin{enumerate}
    \item $\bigcap_{n=1}^\infty \ob{\slam^n} = \ell^2(V)$.
    For any $v = -k \in V$, we have $\slam e_{-k-1} = e_{-k}$. Thus every basis vector is in the range of $\slam$, and consequently in the range of $\slam^n$ for all $n \in \nbb$ (since $\slam$ acts as a surjection on the basis set). Thus the intersection is the whole space $\ell^2(V)$.

    \item $\bigcap_{n=1}^\infty \overline{\ob{|\slam^n|}} = \{0\}$.
    Recall that $\overline{\ob{|\slam^n|}} = \overline{\ob{\slam^{*n}}}$. By \cite[Proposition 3.4.1]{2012-mams-j-j-s}, $\slam^* e_{-k} = e_{-k-1}$. The range of $\slam^{*n}$ is the closed linear span of $\{e_{-k} \colon k \ge n\}$. As $n \to \infty$, the support of any function in the range moves to $-\infty$. The only vector belonging to all such subspaces is the zero vector.
\end{enumerate}
Thus, $\slam$ is a type II centered weighted shift.
\end{exa}
We leave the proof of the following proposition to the reader.
\begin{pro}
Let $\tcal=(V,E)$ be a directed tree isomorphic to $\zbb_-$ (see Figure \ref{fig:Z_minus_tree}). Let $\slam \in \ogr{\ell^2(V)}$ be a centered weighted shift on $\tcal$ with nonzero weights. Then $\slam$ is type II centered.
\end{pro}

\section{Acknowledgments}
I wish to express my deep gratitude to Professor Marek Ptak and his wife for their hospitality and support during and beyond preparation of this paper.
\bibliographystyle{amsalpha}

\end{document}